\pdfoutput=1 		
\documentclass[leqno]{article}
\usepackage[normal]{phaine_style}
\usepackage[margin=1.5in]{geometry}
\usepackage{xparse}
\include{standard_macros}

\renewcommand{\lex}{\mathrm{lex}}
\newcommand{\acc}{\mathrm{acc}}
\newcommand{\Funlexacc}{\Fun^{\lex,\acc}}

\newcommand{\qTopGrp}{\mathrm{q}\TopGrp}

\renewcommand{\Fun}{\mathrm{Fun}}
\renewcommand{\cts}{\mathrm{cts}}

\renewcommand{\Functs}{\Fun^{\cts}}

\newcommand{\Ani}{\categ{Ani}}
\renewcommand{\Spc}{\Ani}
\DeclareMathOperator{\Cond}{Cond}
\newcommand{\ICond}{\categ{Cond}}
\newcommand{\FunctsInt}{\mathbf{Fun}^{\cts}}
\newcommand{\CondAni}{\Cond(\Ani)}
\newcommand{\CondSet}{\Cond(\Set)}

\newcommand{\CondCat}{\Cond(\Catinfty)}

\newcommand{\CondGrp}{\Cond(\Grp)}

\newcommand{\Comp}{\categ{Comp}}

\DeclareMathOperator{\FinSet}{\Set_{\fin}}

\DeclareMathOperator{\ProFin}{\Pro(\FinSet)}

\newcommand{\Extr}{\categ{Extr}}
\newcommand{\Extrop}{\Extr^{\op}}
\newcommand {\CW}{
	\categ{CW}}
\newcommand {\CG}{
	\categ{CG}}

\DeclareMathOperator{\wc}{wc}
\newcommand{\Affwc}{\Aff{}^{\kern0.15em\wc}}
\newcommand{\Affwcop}{\Aff{}^{\kern0.15em\wc,\op}}

\newcommand{\aff}{\ensuremath{\textup{aff}}}

\newcommand{\RTop}{\categ{RTop}}

\newcommand{\cond}{\mathrm{cond}}

\newcommand{\TopGrp}{\Top\Grp}

\newcommand{\picond}{\uppi^{\cond}}

\newcommand{\pizerocond}{\picond_0}

\renewcommand{\Setfin}{\Set_{\fin}}

\newcommand{\ProSetfin}{\Pro(\Setfin)}

\newcommand{\ProAni}{\Pro(\Ani)}

\renewcommand{\Spec}{\mathrm{Spec}}

\NewDocumentCommand{\multgrp}{o}{\mathbb{G}\IfValueTF{#1}{_{\mup, #1}}{_{\mup}}}

\newcommand{\Zaraff}[1]{\mathrm{Zar}{}_{#1}^{\kern0.1em\aff}}

\newcommand{\ProEt}[1]{\mathrm{Pro\acute{E}t}_{#1}}

\newcommand{\ProEtaff}[1]{\mathrm{Pro\acute{E}t}{}_{#1}^{\kern0.1em\aff}}

\newcommand{\ProZaraff}[1]{\mathrm{ProZar}{}_{#1}^{\kern0.1em\aff}}

\usepackage{interval}

\title{\Large Shape theory for condensed anima}

\author{Catrin Mair}

\date{\today}

\begin{document}

\maketitle


\begin{abstract} 
We give different perspectives on the notion of shape for condensed anima. We prove that it recovers more classical notions of shape not only for all locally compact Hausdorff but also locally contractible spaces. These recovering statements imply and extend comparison results on sheaf and condensed cohomology by Clausen-Scholze and Haine. 
Another homotopy-theoretical direction for condensed anima are their condensed homotopy groups. Connected to this, we give a description of the underlying topological group functor on condensed groups via quasi-topological groups.   
\end{abstract}

\tableofcontents



\section{Introduction: Overview and main results}


In Algebraic Topology, one studies topological spaces by means of algebraic invariants like homotopy and cohomology.
The theory of condensed mathematics suggests that topological spaces should be replaced by condensed sets or anima.
It seems natural to ask how one can translate classical concepts from algebraic topology developed for topological spaces into this setting.  
\subsection*{An overview}
\textit{Condensed sets} are defined as sheaves on the category of totally (or extremally) disconnected compact Hausdorff spaces.
The passage from topological spaces to condensed sets and vice versa is encoded in an adjoint pair of functors 
\[
 \begin{tikzcd}[column sep = huge]
            \Top \arrow[r, shift left=1ex] & \CondSet \period \arrow[l, shift left=.5ex]
\end{tikzcd}
\]
The higher categorical version of $\CondSet$ is the $\infty$-category $\CondAni$ of \textit{condensed anima}.
It combines two flavors of the notion of 'spaces' appearing widely in the literature:
On the one hand, there is the topological direction given by a fully faithful embedding from the category $\CondSet$ of condensed sets constituting the $0$-truncated condensed anima. 
By this, we can regard a topological space, identified with a corresponding condensed set, as an object in $\CondAni$. 
On the other hand, the fully faithful discrete sheaf functor identifies \textit{homotopy types}, classically referred to as \textit{spaces} or \textit{$\infty$-groupoids}, and modeled as objects in the $\infty$-category $\Ani$ of \textit{anima}, with those condensed anima carrying no topological information. 
\begin{equation}\label{Diagram_1}
\begin{tikzcd}[row sep=small]
\CondSet \arrow[hookrightarrow,dr,swap, "\text{topological direction}"] 
 & & \Ani \arrow[hook',dl,  "\text{homotopical direction}"] \\
 &\CondAni& 
\end{tikzcd}
\end{equation}
According to their characteristics, condensed anima belonging to the topological direction are called \textit{static}, while those of the homotopical direction are called \textit{discrete}.
Extending homotopy-theoretic notions defined for topological spaces or anima, respectively, along the two directions allows us to study functorially homotopy-theoretic properties of condensed anima:
\begin{equation}\label{Diagram_2}
\begin{tikzcd}
&\CondAni \arrow[rightarrow, swap, dl, "\text{condensed homotopy groups}"]\arrow[rightarrow, dr, "\text{shape}"]& \\
\CondGrp \arrow[d,swap, "\text{underlying topological group}"]& & \ProAni \\
\TopGrp & &
\end{tikzcd}
\end{equation}
\subsection*{Main results}
In this paper, we specify the functors in Diagram (\ref{Diagram_2}). Note that the \textit{shape functor} on the right already has been studied by the author in their master's thesis \cite{arXiv:2105.07888}.
We extend the results stated there under the use of \textit{shape theory}. We start with giving several perspectives on this functor, all of which are summarized in the following theorem.
\begin{theorem}
The shape functor on condensed anima $$\Shape^c\colon \CondAni \rightarrow \ProAni$$ is uniquely determined as the extension of the restricted shape functor $$\Shape \colon \Extr\subset \Top \xrightarrow{\Sh(\cdot)} \RTop_{\infty}\xrightarrow{\Shape} \ProAni$$
on topological spaces by sifted colimits along the inclusion $\Extr \subset \CondAni$ of extremally disconnected profinite sets.
Equivalently, it can be characterized as 
\begin{itemize}
    \item[(i)] the pro-left adjoint of the discrete functor
    $(\cdot)^{\mathrm{disc}}\colon \Ani \hookrightarrow \CondAni$ (\cref{shape_prohomotopy_agree}),
    \item[(ii)] the composite $\CondAni \xrightarrow{\CondAni_{/\cdot}} \RTop_{\infty}\xlongrightarrow{\quad \Shape \quad} \ProAni$ (\cref{relativeshape}),
    \item[(iii)] the composite $\CondAni \xrightarrow{\Sh_{\et}(\cdot) } \RTop_{\infty}\xlongrightarrow{\quad \Shape \quad} \ProAni$ (\cref{cor:ettopos_cshape}).
\end{itemize}
Here the \'etale $\infty$-topos $\Sh_{\et}(X)\subset \CondAni_{/X}$ of a condensed anima $X$ is the full sub-$\infty$-topos of so-called \emph{\'etale sheaves}, see \cref{def:etalesheaves}.
\end{theorem}
We analyze the extent to which the shape functor on condensed anima recovers homotopy-theoretic notions of a topological space such as its homotopy type, shape or cohomology. Our results extend work by Clausen-Scholze \cite[Theorem 3.2]{Scholze:condensednotes} and Haine \cite[Corollary 4.12]{arXiv:2210.00186} on the comparison of sheaf and condensed cohomology for (locally) compact Hausdorff spaces. We do not consider as general coefficients as Haine does, but expand the class of topological spaces to locally contractible spaces. This basically follows as $\Shape^c\colon \CondAni\rightarrow \ProAni$ is homotopy invariant, see \cref{condensedhomotopyinv}.
\begin{theorem}
 Let $T\in \Top$ be a topological space and $\underline{T}\in \CondAni$ its associated condensed anima via the composed functor
 $$\Top \xrightarrow{T\mapsto \Hom_{\Top}(\cdot,T)} \CondSet \hookrightarrow \CondAni\period $$ 
 \begin{itemize}
     \item[1.] Then the condensed shape $\Shape^c(\underline{T})$ recovers 
        \begin{itemize}
         \item[a)] the shape $\Shape^{\mathrm{post}}(T)$ if the space $T$ is locally compact Hausdorff (\cref{prohoshape}). 
         \item[b)] the fundamental $\infty$-groupoid $\Sing(T)\in \Ani\subset \ProAni$ if $T$ is a locally contractible space (\cref{Locally_contractible_cshape}).
\end{itemize}
     \item[2.] In both cases a) and b), one obtains for all $A\in \Ab\Grp$ natural isomorphisms of sheaf and condensed cohomology (\cref{prop:coho_cond_sheaf})
$$H^n_{\mathrm{sheaf}}(T,A) \xrightarrow{\simeq} H^n_{\mathrm{cond}}(\underline{T},A)\in \Grp.$$
In case b), we further deduce natural isomorphisms with singular cohomology (\cref{cor:coho_cech_cond}):
$$H^n_{\mathrm{sing}}(T,A)\simeq H^n_{\mathrm{cond}}(\underline{T},A)\period $$
\end{itemize}
\end{theorem}
Restriction of the functor $\Top \rightarrow \CondSet$ to group objects admits a left adjoint. This is the \textit{underlying topological group functor} $\CondGrp \rightarrow \TopGrp$ on the left hand side in Diagram (\ref{Diagram_2}). We find a description of this functor via quasi-topological groups $ \mathrm{q}\TopGrp$, i.e., topologized groups such that the inversion map and all the translation maps are continuous.
We use that every quasi-topological group can be refined to a topological group via the left adjoint $\tau\colon \mathrm{q}\TopGrp\rightarrow \TopGrp$ of the fully faithful embedding $\TopGrp\hookrightarrow \mathrm{q}\TopGrp$.
\begin{proposition}[(\cref{prop:underlying_top})] Restriction of the left adjoint of the functor $\Top \rightarrow \CondSet$, namely the functor $$(\cdot)^{\Top}\colon\CondSet \rightarrow \Top, \quad X \mapsto X(*)^{\Top}$$ assigning to every condensed set an underlying topological space, to group objects lands in the category of quasi-topological groups, i.e., it restricts to a functor
$\CondGrp \rightarrow \qTopGrp.$
The underlying topological group functor $(\cdot)^{\TopGrp}\colon \CondGrp\rightarrow \TopGrp$, which is defined as the left adjoint of $\TopGrp\rightarrow \CondGrp$, is given as the composition 
\begin{align*}
(\cdot)^{\tau}\colon\CondGrp\rightarrow \mathrm{q}\TopGrp\xrightarrow{\tau} \TopGrp.
\end{align*}
\end{proposition}
\subsection*{Conventions}\label[subsection]{intro_subsec:notational_conventions}
\subsubsection*{Set-theoretic conventions}
Whenever working with condensed mathematics, there arise set-theoretic issues from the fact that defining sites are large. We follow the approach used in the setting of \textit{pyknotic mathematics} \cites{pyknoticI}{exodromy}  to deal with these. More on this can be found in \cref{ignorsetiss}. As an upshot: We will neglect set-theoretic issues in the sense that we will implicitly work with $\kappa$-small objects or functors for strongly inaccessible cardinals $\kappa$. We use results on condensed/pyknotic mathematics from \cites{pyknoticI}{Exodromy}{Scholze:condensednotes}{Scholze:analyticnotes} without any further notice on set-theory. 
\subsubsection*{On higher category theory}
We follow the conventions by Lurie in \cite{HTT}, \cite{Kerodon} and \cite{SAG} regarding aspects of higher category theory like $\infty$-categorical definitions and constructions.
The prefix "$\infty$" might be omitted if it is clear that all objects are considered in the $\infty$-categorical setting.
Moreover, we make the following implicit assumptions:
\begin{itemize}
\item Whenever working with a $1$-category in the $\infty$-setting, it will be implicitly identified with its corresponding $\infty$-category by the nerve construction.
\item If there exists a fully faithful embedding of $\infty$-categories 
$$\iota\colon \Ccal \hookrightarrow \mathcal{D}\comma $$ we will usually identify an object $X \in \Ccal$ with its image under $\iota$ and will refer to the corresponding object in $\mathcal{D}$ with $X$ as well.
\item For an $\infty$-category $\Ccal$, we denote the $\infty$-category of pro-objects on $\Ccal$, which is naturally identified with the opposite of the \category of left exact accessible functors $ \fromto{\Ccal}{\Ani} $, by
\begin{equation*}
        \Pro(\Ccal) \equivalent \Funlexacc(\Ccal,\Ani)^{\op} \comma
\end{equation*}
see \cite[\SAGthm{Definition}{A.8.1.1} \& \SAGthm{Proposition}{A.8.1.6}]{SAG}. We will implicitly identify an object in $\Ccal$ with its image under the 'Yoneda embedding' meaning the fully faithful functor $j_{\Ccal}\colon \Ccal\hookrightarrow\Pro(\Ccal)$. 
\end{itemize}
\subsubsection*{Notational conventions}
Throughout this paper, there appear several ($\infty$-)categories. 
We fix the following (standard) notation.
\begin{enumerate}
    \item We write $\Set$ for the category of sets and $\Set^{\mathrm{fin}}$ for the category of finite sets.
    \item We write $\Grp$, respectively,  $\Ab\Grp$ for the category of (abelian) groups. By $\TopGrp$ we denote the category of topological groups.
    \item  We write $\Top$ for the category of topological spaces, $\Comp$ for the category of compact Hausdorff spaces, $\ProFin$ for the category of totally disconnected compact Hausdorff spaces (=profinite sets) and $\Extr$ for the category of extremally disconnected profinite sets. By $\CW$ we denote the category of CW-spaces and by $\CG$ the category of compactly generated spaces using the definition recalled in \cref{def:compactlygenerated}. 
	\item We write $\Cat$, respectively, $ \Catinfty $ for the large ($\infty$-)category of small ($\infty$-)categories, and use $ \Ani \subset \Catinfty $ for the full subcategory spanned by the anima (\groupoids/ spaces).
    By $\Ani_{< \infty}\subset \Ani$ we denote the subcategory of truncated anima.
    \item We write $\Cond(\Ccal)$ for the ($\infty$-)category of condensed objects on some ($\infty$-)category $\Ccal$.
    \item We write $\Fun(\Ccal, \mathcal{D})$ for the $\infty$-category of functors between $\infty$-categories $\Ccal$ and $\mathcal{D}$ and $\Functs(\Ccal, \mathcal{D})$ for the $\infty$-category of functors between condensed $\infty$-categories $\Ccal$ and $\mathcal{D}$.
    \item We write $\RTop_{\infty}$ for the $\infty$-category of $\infty$-topoi with \textit{geometric morphisms}, i.e., a morphism $f\colon \mathcal{X} \rightarrow \mathcal{Y}$ between two $\infty$-topoi is given by a functor $f_*\colon \mathcal{X} \rightarrow \mathcal{Y} $  admitting a left exact left adjoint $f^* \colon \mathcal{X} \rightarrow \mathcal{Y}$. We refer to the left adjoint $f^*$ as an \textit{algebraic morphism}. The $\infty$-category of morphisms between two $\infty$-topoi $\mathcal{X}$ and $\mathcal{Y}$ in $\RTop_{\infty}$ is denoted by $\Fun_*(\mathcal{X},\mathcal{Y})$. It is the opposite category of the $\infty$-category of algebraic morphisms denoted by $\Fun^*(\mathcal{X},\mathcal{Y})$.
     \item For a site $\Ccal$ (with a Grothendieck topology $\tau$) we write $\Sh(\Ccal)$ for the $\infty$-category of sheaves of anima on $(\Ccal,\tau)$, $ \Shhyp(\Ccal) \subset \Sh(\Ccal)$ for the subcategory of hypercomplete objects (i.e., sheaves of anima satisfying descent with respect to hypercoverings), and $\Sh(\Ccal, \Set)$ for the category of sheaves of sets on $(\Ccal, \tau)$, i.e., the subcategory of $0$-truncated objects in $\Sh(\Ccal)$.

\end{enumerate}

\subsection*{Acknowledgments}\label[subsection]{intro_subsec:acknowledgments}
This paper emerged from the first chapter of my PhD thesis \cite{CatrinsThesis} and generalizes parts of my master's thesis \cite{arXiv:2105.07888}. 
In an earlier draft of this preprint, there was included a false recovering theorem for the shape of paracompact compactly generated spaces which we removed here. We got aware of this error through the use of a LLM in another context. After updating this version, we had an AI run another check on everything which lead to minor corrections.\\
I would like to once again thank my advisor, Torsten Wedhorn, for all his support, as well as the referees Clark Barwick and Timo Richarz. Special thanks go to Peter Scholze for answering questions on my master's thesis topic which laid the foundation for this paper.
I gratefully acknowledge support by Deutsche For\-schungs\-ge\-mein\-schaft (DFG, German Research Foundation) through the Collaborative Research Centre TRR 326 Geometry and Arithmetic of Uniformized Structures, project number 444845124, and Germany's Excellence Strategy EXC 2044/2 - 390685587, Mathematics Münster: Dynamics-Geometry-Structure.


\section{Preliminaries on shape theory and condensed anima}
We provide some background theory on the notion of shape for $\infty$-topoi and topological spaces as well as on the basics of condensed mathematics.
\subsection{Recollection on shape theory}\label[subsection]{sec:shape_theory}
As further references for shape theory consult \cite[Section 7.1.6]{HTT} and \cite[Section 4.2]{Exodromy}.
\begin{recollection}
The $\infty$-category $\Ani$ of \textit{anima} is the $\infty$-categorical version of $\Set$. For the notion of \textit{shape}, we need to work with its pro-version: The $\infty$-category $\ProAni$ of pro-anima is defined as the sub-$\infty$-category of finite limit preserving accessible $\infty$-functors in $\Fun(\Ani, \Ani)^{\op}$. Its objects can be identified with cofiltered systems of anima \cite[Remark A.8.1.5]{SAG}.
\end{recollection}
\begin{construction}{\textbf{Pro-left adjoint}}\\ \label[construction]{proleftadjoint}For every $\infty$-topos $\mathcal{X}$, there exists a (up to a contractible choice) unique geometric morphism $g \colon \mathcal{X} \rightarrow \Ani$ consisting of a pair of adjoint functors
\[
 \begin{tikzcd}[column sep = huge]
            \mathcal{X} \arrow[r, shift left=1ex, "g_*"] & \Ani\comma \arrow[l, shift left=.5ex, "g^*"]
           \end{tikzcd}
   \] such that the pullback $g^*$ is a left exact left adjoint of the pushforward $g_*$. 
The pullback functor $g^*$ admits a pro-left adjoint $g_!\colon \mathcal{X} \rightarrow \Pro(\Ani)$
\begin{equation*}
\begin{tikzcd}[row sep=huge]
\mathcal{X} \arrow[rr, shift left=1ex, "g_*"] \arrow[dr, "g_!", dotted, swap]&
& \Ani\comma  \arrow[ll, shift left=.5ex, "g^*"]\arrow[dl, hook, "j_{\Ani}"]
\\ 
& \Pro(\Ani) &  
\end{tikzcd}
\end{equation*}
where $j_{\Ani}$ denotes the Yoneda embedding.
The functor $g_!$ is the left adjoint of the limit preserving pro-extension of $g^*$ sending a cofiltered system of anima to the limit of its evaluation under $g^*$
$$\Pro(g^*)\colon\Pro(\Ani)\rightarrow \mathcal{X}, \quad \{A_i\}_{i \in I}\mapsto \lim_{i \in I} g^*(A_i)\period $$
The image of an object $X\in\mathcal{X}$ under $g_!$ is the finite limit preserving functor
$$\Ani \rightarrow \Ani, \quad A \mapsto \Hom_{\mathcal{X}}(X, g^*(A))\period $$
\end{construction}
\begin{definition}{\textbf{Shape of an $\infty$-topos}}\\
The \textit{shape} of an $\infty$-topos $\mathcal{X}$ is defined as the image 
$$\Pi_{\infty}(\mathcal{X})\colonequals g_!(\mathbf{1}_\mathcal{X})\in \Pro(\Ani)$$
of the terminal object $\mathbf{1}_{\mathcal{X}}$ in $\mathcal{X}$ under the pro-left adjoint $g_!\colon \mathcal{X} \rightarrow \Pro(\Ani)$.
\end{definition}
\begin{remark}
The unit map $\mathrm{id_{\mathcal{X}}} \rightarrow h_*h^*$ of a geometric morphism $h\colon \mathcal{Y} \rightarrow \mathcal{X}$ 
of $\infty$-topoi
induces a map of shapes $\Pi_{\infty}(\mathcal{Y}) \rightarrow \Pi_{\infty}(\mathcal{X})\in \ProAni$.
By this, 
we obtain a functor $$\Pi_ {\infty}\colon\RTop_{\infty}\rightarrow \Pro(\Ani)$$ on the $\infty$-category $\RTop_{\infty}$ of $\infty$-topoi (with geometric morphisms) assigning the shape.
The assignment is colimit preserving and $\Shape$ admits a right adjoint~\cite[Remark 7.1.6.15]{HTT}.  
\end{remark}
\begin{remark}{\cite[Remark 6.1.4]{Exodromy}}\label[remark]{fullyshapeequiv}
Every geometric morphism $h\colon \mathcal{Y} \rightarrow \mathcal{X}$ with a fully faithful pullback morphism $h^*\colon \mathcal{X} \hookrightarrow \mathcal{Y}$ is a \textit{shape equivalence}, i.e., it induces an equivalence of pro-anima $\Shape(\mathcal{X})\simeq \Shape(\mathcal{Y})$.
\end{remark}
\begin{definition}{\textbf{Shape of a topological space}}\\
Let $T$ be a topological space. We define its \textit{shape} as the shape of the $\infty$-topos $\Sh(T)\colonequals \Sh(\mathrm{Ouv(T)})$ of sheaves of anima (on the site of open subsets) over $T$
$$\Shape(T)\colonequals \Shape(\Sh(T))\period $$
\end{definition}
\begin{example}\label[example]{Shape_CW}
If $T$ is a CW-complex, the shape agrees with the fundamental $\infty$-groupoid $\Sing(T)$ of $T$ \cite[Example 2.4]{MR3763287}
$$\Pi_{\infty}(T)= \Sing(T) \in \Ani\period $$
For an arbitrary topological space, there is always a map 
$$\Sing(T) \rightarrow \Pi_{\infty}(T) \in \ProAni\period $$ 
\end{example}
It is not clear to us whether the shape $\Shape(T)$ recovers the classical homotopy type $\Sing(T)$ for all locally contractible spaces if we work with sheaves of anima over $T$.
It certainly becomes true if one takes hypersheaves instead, i.e., works with the hypercompletion $\Shhyp(T)$ of the $\infty$-topos $\Sh(T)$. For us, a locally contractible space is defined as follows.
\begin{definition}{\textbf{Locally contractible spaces}}\\
A topological space $T$ is called \emph{locally contractible} if it admits a basis by contractible open subsets.    
\end{definition}
\begin{remark}\label[remark]{locallycontrshape}
For a locally contractible space $T$, the shape of the $\infty$-topos of hypersheaves $\Shhyp(T)$ over $T$ coincides with the fundamental $\infty$-groupoid $\Sing(T)$~\cite[Proposition~2.2.7]{MR4367219} or \cite[Corollary~3.4]{MR4651175}.
In particular, in the case of $T$ being a CW-space we not only have an equivalence of shapes $$\Shape(T)=\Sing(T)=\Shape(\Shhyp(T))\comma $$ but the $\infty$-topos $\Sh(T)$ itself is already hypercomplete \cite{MO:168526}.
\end{remark}
For an arbitrary $\infty$-topos $\mathcal{X}$, hypercompleteness is not automatic. However, from a homotopy-theoretic point of view, it does not make that much of a difference as the shapes of $\mathcal{X}$ and its hypercompletion $\mathcal{X}^{\mathrm{hyp}}$ are equivalent only in a slightly weaker sense.
\begin{construction}
The \textit{pro-truncated shape} functor
$$\Pi_{<\infty}\colon \RTop_{\infty}\rightarrow \Pro(\Ani_{<\infty})$$
is defined as the composite 
$$\RTop_{\infty}\xrightarrow{\Shape} \ProAni \xrightarrow{\Pro(\tau_{<\infty})} \Pro(\Ani_{<\infty})$$ of the shape functor with the unique cofiltered limit preserving pro-extension of the functor
$$\tau_{< \infty}\colon \Ani \rightarrow \Pro(\Ani_{<\infty}), \quad A \mapsto \{\tau_nA\}_n$$
assigning the \textit{Postnikov tower} to an anima.
\end{construction}
\begin{notation}
    Let $\mathcal{X}$ be an $\infty$-topos. We denote by $\mathcal{X}^{\mathrm{post}}$ its \textit{Postnikov completion}, compare \cite[Subsection A.7.2]{SAG}, and by $\mathcal{X}^{\mathrm{hyp}}$ its \textit{hypercompletion}, compare \cite[Section 6.5.2]{HTT}. As the Postnikov completion of an $\infty$-topos coincides with the Postnikov completion of its hypercompletion and the hypercompletion is a reflective sub-$\infty$-category, we have a chain of geometric morphisms
    $$\mathcal{X}^{\mathrm{post}} \rightarrow \mathcal{X}^{\mathrm{hyp}} \hookrightarrow \mathcal{X} \period$$
\end{notation}
\begin{lemma}\label[lemma]{protruncation_shapeequiv}
Let $\mathcal{X}$ be an $\infty$-topos. The geometric morphisms
$$\mathcal{X}^{\mathrm{post}} \rightarrow \mathcal{X}^{\mathrm{hyp}} \hookrightarrow \mathcal{X}$$
induce equivalences under $\Pi_{<\infty}$, i.e., the shapes of $\mathcal{X}$, the hypercompletion $\mathcal{X}^{\mathrm{hyp}}$ and the Postnikov completion $\mathcal{X}^{\mathrm{post}}$ agree up to pro-truncation.
\end{lemma}
\begin{proof}
For the hypercompletion, see~\cite[Example~4.2.8]{Exodromy}. The statement on Postnikov completions follows by~\cite[Theorem~A.7.2.4]{SAG} which states that the Postnikov completion functor $e^*\colon \mathcal{X} \rightarrow \mathcal{X}^{\mathrm{post}}$ of an $\infty$-topos $\mathcal{X}$ induces an equivalence on $n$-truncations for all $n\geq 0$. 
\end{proof}
\begin{definition}
Let $\mathcal{X}$ be an $\infty$-topos. We write
$$\Pi_{\infty}^{\mathrm{post}}(\mathcal{X})\colonequals \Shape(\mathcal{X}^{\mathrm{post}}) \text{ and } \Pi_{\infty}^{\mathrm{hyp}}(\mathcal{X})\colonequals \Shape(\mathcal{X}^{\mathrm{hyp}}) $$
for the shape of the Postnikov completion, resp., hypercompletion, and for a topological space $T$
$$\Pi_{\infty}^{\mathrm{post}}(T)\colonequals \Shape(\Sh^{\mathrm{post}}(T))\text{ and } \Pi_{\infty}^{\mathrm{hyp}}(T)\colonequals \Shape(\Sh^{\mathrm{hyp}}(T))\period$$
\end{definition}
\begin{remark}
We can assign homotopy pro-groups to objects in $\ProAni$ by extending simplicial homotopy groups from $\Ani$.
As these groups only depend on some $n$-truncation, the shapes of the $\infty$-topoi $\mathcal{X}$, $\mathcal{X}^{\mathrm{hyp}}$ and $\mathcal{X}^{\mathrm{post}}$ have the same homotopy pro-groups.
\end{remark}
Since we mostly study a pro-anima via its homotopy pro-groups anyway, referring to the last remark, we will often work with one of these completions.
\begin{example}\label[example]{shape_profinset}
If $S$ is a profinite set, then the $\infty$-topos $\Sh(S)$ is already Postnikov complete as it has locally finite homotopy dimension, see \cite[Example 1.28]{arXiv:2210.00186} (also \cref{ex:profinitepostnikov} later).
In particular, since Postnikov complete objects are always hypercomplete, one has $$\Sh(S)=\Sh^{\mathrm{post}}(S)=\Shhyp(S)\period $$ The shape $\Shape(S)$ is given by $S$ itself viewed as an object in $\ProAni$ via the fully faithful embedding of pro-\categories $\ProFin \subset \ProAni$ which is induced by extending $\Setfin\subset \Ani$ along cofiltered limits to pro-objects.
This can be seen by combining \cite[Proposition E.2.2.1, Theorem E.2.4.1, Example E.2.4.6]{SAG}.
\end{example}
Restriction of the functor $\Pi_{\infty}\colon \RTop_{\infty}\rightarrow \ProAni$ to the category $\Top$ of topological spaces behaves nicely with respect to homotopy equivalences of topological spaces.
\begin{proposition}\label[proposition]{homotopyinv}
The induced functor $\Pi_{\infty}\colon \Top \xrightarrow{\Sh(\cdot)} \RTop_{\infty}\xrightarrow{\Shape(\cdot)}\ProAni$, which assigns the shape to a topological space, is homotopy invariant, i.e., it inverts homotopy equivalences of topological spaces.
\end{proposition}
\begin{proof}
It is enough to show that $\Shape$ inverts the projection map $T\times [0,1] \rightarrow T$ for all $T \in \Top$. The argument for this is quite standard:
Let $H\colon T \times I\rightarrow Y$ be an arbitrary homotopy between two maps $a=H\circ i_0$ and $b=H\circ i_1$. The projection map being inverted especially implies that its homotopy inverses $i_0$ and $i_1$ are identified.
Thus also $a$ and $b$ are identified.
This indeed shows that in the case of a homotopy equivalence $f\colon T \rightarrow Y$ with homotopy inverse $g\colon Y \rightarrow T$ we can identify $g\circ f$ and $f \circ g$ with the identity morphisms on $T$ and $Y$. 
As every projection $T \times [0,1] \rightarrow T$ of topological spaces induces a geometric morphism with fully faithful pullback morphism on the $\infty$-topoi of sheaves \cite[Lemma 4.3 and Remark 4.5]{MR4651175}, the claim then follows from \cref{fullyshapeequiv}.
\end{proof}
\subsection{Background on condensed mathematics}

We review  condensed sets, condensed anima and condensed $\infty$-categories. As additional sources consult \cite{pyknoticI} and \cite{Scholze:condensednotes}. Note \cref{ignorsetiss} for a comment on set-theoretic issues that come along with definitions in the world of condensed mathematics.
\subsubsection{Condensed sets and condensed anima}
The basic objects in condensed mathematics are \textit{condensed sets}. Originally, these are defined as sheaves of sets on the \textit{pro-\'{e}tale site} $\ProEt{*}$ of a point, i.e., the spectrum of a separably closed field. See \cite{MR3379634} for an introduction to the \textit{pro-\'etale topology of schemes.} The category of condensed sets is denoted by $$\CondSet \colonequals \Sh(\ProEt{*},\Set) \period$$
Sheaves on $\ProEt{*}$ can be identified with sheaves on certain subcategories of topological spaces. The notion of condensed sets extends in the $\infty$-categorical world to the notion of \textit{condensed anima}, denoted by $\CondAni$, by considering \textit{hypersheaves} with values in the $\infty$-category $\Ani$ of \textit{anima} (or spaces).
\begin{notation}\label[notation]{not:topo_notions}
We denote by $\ProFin$ the category of \textit{profinite sets}. By Stone duality, the category $\ProFin$ can be identified with the full subcategory of totally disconnected spaces in the category $\Comp$ of \textit{compact Hausdorff spaces}.
By a result by Gleason, stated in \cite[p. III 3.7]{MR861951}, the projective objects in $\Comp$ are the \textit{extremally disconnected} profinite sets $\Extr \subset \Comp$.
\end{notation}
\begin{recollection}\label[recollection]{rec:definingsites}
There are at least four different defining sites for the $\infty$-category $\CondAni$ of \textit{condensed anima} defining the same $\infty$-topos of hypercomplete $\infty$-sheaves:
\begin{itemize}
    \item[(i)] The pro-\'{e}tale site $\ProEt{\Spec(k)}$ for $k$ any separably closed field,
    \item[(ii)] the site $\Comp$ of compact Hausdorff spaces,
    \item[(iii)] the site $\ProFin$ of profinite sets, 
    \item[(iv)] the site $\Extr$ of extremally disconnected profinite sets.
\end{itemize}
Here, coverings for the sites (ii)-(iv) are given by jointly surjective families of continuous maps of topological spaces.
All of these sites define the same $\infty$-topos on hypercomplete $\infty$-sheaves: Every compact Hausdorff space admits a surjection from an extremally disconnected compact Hausdorff space, namely from the Čech--Stone compactification of its underlying discrete space. Moreover, $\ProFin$ identifies with the subcategory of affine schemes in $\ProEt{\Spec(k)}$. Thus there are hierarchies of bases 
$ \Extr \subset \ProFin \subset \Comp \text{ and } \ProFin \subset \ProEt{\Spec(k)}.$
By \cite[Corollary A.7]{arXiv:2001.00319}, restriction of sites induces equivalences of hypercomplete \topoi
    \begin{equation}\label{eq:equivalent_descriptions_of_condensed_anima}
        \Shhyp(\ProEt{\Spec(k)}) \equivalence \Shhyp(\ProSetfin) \equivalence \Shhyp(\Extr) \period
    \end{equation}
    The \topos $\CondAni$ of \textit{condensed anima} is any of the equivalent \topoi \eqref{eq:equivalent_descriptions_of_condensed_anima}.
\end{recollection}
\begin{recollection}\label[recollection]{rec:fully_faithful_condset_to_condani}
    The category $\CondSet$ of condensed sets embeds fully faithfully into the $\infty$-category $\CondAni$ of condensed anima by post-composition of functors with the fully faithful embedding $v\colon\Set \hookrightarrow \Ani$. The left adjoint is given by post-composition with the $0$-truncation functor $\pi_0\colon \Ani \rightarrow \Set$. In other words, the objects of $\CondAni$ in the essential image of $\CondSet$ are exactly the $0$-truncated objects. They are called \textit{static} condensed anima.
\end{recollection}
\begin{recollection}\label[recollection]{rec:fin_prod_pres_presheaves}
On extremally disconnected profinite sets, hypersheafification can be omitted and the $\infty$-sheaf condition simplifies. More precisely, a presheaf $ F $ on $ \Extr $ is a hypersheaf if and only if $ F $ carries finite disjoint unions to finite products, i.e., there is an identification
    \begin{equation*}
       \CondAni \simeq \Shhyp(\Extr) \equivalent \Funcross(\Extrop,\Ani)\comma
    \end{equation*}
    where the right hand site denotes finite product-preserving functors. It is a consequence of $\Extr$ being a subsite of projective objects and widely used in the literature on condensed mathematics. A detailed discussion can be, for example, found in \cite[Section A.3.5]{CatrinsThesis}.
\end{recollection}
\begin{recollection}
The \textit{animation} of an ordinary category is defined as the $\infty$-category freely generated under sifted colimits by compact projective objects, see \cite[Section 11.1]{Scholze:analyticnotes}. For example, the \category $\Ani$ of anima (or animated sets) is the animation of the category $\Set$ of sets where the compact projectives are exactly the finite sets, see \cite[Example 11.5]{Scholze:analyticnotes}.
\end{recollection}
\begin{definition}{\textbf{Animation of condensed sets}}\label[definition]{anicond}\\
The animation $\Ani(\CondSet)$ of the category $\CondSet$ is the $\infty$-category freely generated under sifted colimits by the category of extremally disconnected profinite sets $\Extr$ as these are the compact projective objects in $\CondSet$, see \cites[Chapter 3]{arXiv:2105.07888}[Example 11.3 (4)]{Scholze:analyticnotes}. 
Equivalently, it can be defined as the full $\infty$-subcategory of functors in
$$\mathrm{Fun}(\Extr^{\mathrm{op}}, \Ani)$$ generated under sifted colimits by the Yoneda image.
\end{definition}
\begin{remark}
From the description of $\CondAni$ in \cref{rec:fin_prod_pres_presheaves}, it follows that sifted colimits in $\CondAni$ can be computed in the presheaf category $ \Fun(\Extr^{\op},\Ani) $ and that there is an equivalence of $\infty$-categories
$$\CondAni \simeq \Ani(\CondSet)\comma $$
also see \cite[Definition 11.7 and Lemma 11.8]{Scholze:analyticnotes}.
\end{remark}
Condensed sets can be seen as a replacement of topological spaces in the following sense.
\begin{example}\label[example]{ex:top}
To a topological space $T\in \Top$, we associate a condensed set $\underline{T}\in \CondSet$ via the restricted Yoneda embedding
\begin{align*}
\underline{T}\colon\Extr^{\mathrm{op}} &\rightarrow \Set, \
S \mapsto \Hom_{\Top}(S,T)\period
\end{align*}
If $T$ is a topological group, every set $\Hom_{\Top}(S,T)$ inherits the group structure and $\underline{T}$ is a group object in $\CondSet$.
Speaking of a topological space in the context of condensed sets, we will usually mean the condensed set associated to the topological space under this functor.
\end{example}
\begin{remark} Note that there are set-theoretic issues in the definition of condensed objects, see \cref{ignorsetiss}. For that reason the construction in \cref{ex:top} will not give a condensed set in the sense of Clausen-Scholze \cite{Scholze:condensednotes} if the space $T$ is not $T_1$. This is also commented in \cite[Section 0.3]{pyknoticI} and \cite[Warning 2.14, Proposition 2.15]{Scholze:condensednotes}.
\end{remark}
\begin{remark}\label[remark]{def:compactlygenerated}
The functor $\Top \rightarrow \CondSet$ is not fully faithful in general but when restricted to the full subcategory $\CG\subset \Top $ of compactly generated spaces, see \cite[Proposition 1.7]{Scholze:condensednotes}. Here, \textit{compactly generated} is defined as follows: A topological space $T$ is compactly generated if a map $f \colon T \rightarrow Y$ of topological spaces is continuous if and only if the composite $S \rightarrow T \rightarrow Y$ is continuous for all
compact Hausdorff spaces $S\in \Comp$ mapping to $T$.
\end{remark}
The $\infty$-category $\CondAni$ combines the topological space direction of $\CondSet$ with the homotopy-theoretic direction of $\Ani$.
\begin{recollection}\label[recollection]{rec:discrete_condani}
The $\infty$-category $\Ani$ embeds fully faithfully into $\CondAni$ by the pullback morphism $(\cdot)^{\mathrm{disc}}\colon \Ani \hookrightarrow \CondAni$ of the unique geometric morphism to the (terminal) $\infty$-topos $\Ani$. 
We refer to condensed anima in the essential image as \textit{discrete} condensed anima.
Its right adjoint is the global sections functor $\mathrm{ev}_*\colon \CondAni \rightarrow \Ani$ given by $A\mapsto A(*)$.
\end{recollection}
\subsubsection{Condensed objects in \texorpdfstring{$\infty$}{infinity}-categories}
The description of condensed anima in \cref{rec:fin_prod_pres_presheaves} extends to a general definition of condensed objects of some \category $\Ccal$.
\begin{definition}\label[definition]{def:condobj}
Let $\Ccal$ be an $\infty$-category that admits all finite products. We denote the $\infty$-category of finite product-preserving presheaves $\fromto{\Extr^{\op}}{\Ccal}$ by
$$\Cond(\Ccal)\colonequals \Fun^{\times}(\Extr^{\op}, \Ccal)$$ and refer to its objects as \textit{condensed objects of $\Ccal$}.
\end{definition}
We mainly apply this definition to the ($\infty$-)categories $\Ccal=\Set,\Grp, \Ani$, $\Cat$ and $\Catinfty$.
\begin{warning}\label[warning]{ignorsetiss}
The definition of condensed objects comes with set-theoretic problems due to the size of the category $\Extr$ of extremally disconnected profinite sets. The considered functor categories are not locally small. We follow the suggestion of \cite{pyknoticI} to handle these issues: 
We implicitly fix two strongly inaccessible cardinals $ \delta < \epsilon $ and think about $\Cond(\Ccal)$ as hypersheaves of $\epsilon$-small objects in $\Ccal$ on $\delta$-small objects in $\Extr$. However, as these choices do not affect our results, we will not further comment on set-theoretic issues.
\end{warning}
We have a general version of \cref{rec:discrete_condani} for condensed objects of some $\infty$-category $\Ccal$.
\begin{recollection}\label[recollection]{rec:discrete_functor}
For every $\infty$-category $\Ccal$ such that $\Cond(\Ccal)$ is defined, there is an adjoint pair of functors 
\[
 \begin{tikzcd}[column sep = huge]
           \Cond(\Ccal) \arrow[r, shift left=1ex] & \Ccal \arrow[l, hook',shift left=.5ex, swap]
           \end{tikzcd}
   \]
given by the constant sheaf and global sections functors.
For an object $X\in \Ccal$, the image under $\Ccal \hookrightarrow \Cond(\Ccal)$ is the \textit{discrete condensed object} $X^{\mathrm{disc}}$ given by the assignment
 $$S=\{S_i\}_{i\in I}\mapsto X^{\mathrm{disc}}(S)\colonequals \colim_{i \in I^{\op}} X^{S_i}\comma $$
 where $X^{S_i}$ denotes the product $\prod_{S_i}X$.
 The right adjoint $\Cond(\Ccal) \rightarrow \Ccal, \ X \mapsto X(*)$ sends a condensed object of $\Ccal$ to its \textit{underlying object} in $\Ccal$.
\end{recollection}
The following generalizes the adjoint pair of functors of \cref{rec:fully_faithful_condset_to_condani}.
\begin{recollection}\label[recollection]{rec:adjointcondensed}
If $ \Dcal $ is another \category with finite products and $ F \colon \fromto{\Ccal}{\Dcal} $ is a finite product-preserving functor, we write
    \begin{equation*}
        F^{\cond} \colon \fromto{\Cond(\Ccal)}{\Cond(\Dcal)}
    \end{equation*}
    for the functor given by post-composition with $ F $.
    If $ F \colon \fromto{\Ccal}{\Dcal} $ admits a right adjoint $ G $, then $ G^{\cond} $ is right adjoint to $ F^{\cond} $.
\end{recollection}
For $\Ccal=\Catinfty$, we obtain the $\infty$-category $\CondCat$ of \textit{condensed $\infty$-categories}.
\begin{example}
	We can define condensed \categories $ \ICond(\Ani) $ and $ \ICond(\Set) $ by the assignments
	\begin{equation*}
		S \mapsto \CondAni_{/S} \andeq S \mapsto \Cond(\Set)_{/S} \period
	\end{equation*}
\end{example}
Functors of condensed $\infty$-categories are given as follows.
\begin{recollection}\label[recollection]{rec:contfunc}
Let $\Ccal$ and $\mathcal{D}$ be condensed $\infty$-categories. 
The $\infty$-category $\Functs(\Ccal,\mathcal{D})$ of \textit{continuous functors between condensed $\infty$-categories} is defined as the end of the bifunctor $$(\Extr^{\mathrm{op}})^{\mathrm{op}} \times \Extr^{\mathrm{op}} \rightarrow \Ccal, \ S\times S' \mapsto \Fun(\Ccal(S), \mathcal{D}(S'))\comma $$ 
compare \cite[Definition 13.3.16]{Exodromy}. As proven in \cite[Proposition 2.3]{MR3518559}, this is equivalent to the $\infty$-category $\mathrm{Nat}(\Ccal,\mathcal{D})$ of natural transformations of condensed $\infty$-categories.
We can think of the objects  of $\Functs(\Ccal,\mathcal{D})$ as tuples of functors $\Ccal(S)\rightarrow \mathcal{D}(S)$ indexed by $S\in \Extr$ satisfying certain compatibility conditions. 
\end{recollection}
\subsubsection{Condensed homotopy groups of condensed anima}
For the definition of condensed homotopy groups, note that we have an identification on pointed objects $\CondAni_*=\Cond(\Ani_*)$, see \cite[Remark 2.4.6]{pyknoticI}.
Further recall that for all $n\geq 1$ the simplicial homotopy group functors 
$\pi_n\colon \Ani_* \rightarrow (\categ{Ab})\Grp$ behave well with finite products and thus are compatible with condensed objects by postcomposition, see \cref{rec:adjointcondensed}.
\begin{definition}{\textbf{Condensed homotopy groups}}\label[definition]{def:condensed_homotopy_groups}\\
We define the \textit{condensed homotopy group functors} on pointed condensed anima
$$\pi_n^{\cond}\colon \CondAni_*\rightarrow \CondGrp$$
by level-wise postcomposition with the simplicial homotopy group functors on pointed anima
$\pi_n\colon\Ani_*\rightarrow \Grp,$ for all $n\geq 1$.
\end{definition}
In other words, the \textit{$n$-th condensed homotopy group} $\pi_n^{\cond}(X,x)$ of a pointed condensed anima $(X,x\colon \ast \rightarrow X)$ is the condensed group with value $$\pi_n(X(K), x(K))\in \Grp$$ at every $K\in \Extr$.
For all $n\geq 2$, it is even a condensed abelian group by the corresponding statement for simplicial homotopy groups.
\begin{example}{\textbf{Discrete condensed anima have discrete homotopy groups}}\\
Let $(A,a)\in \CondAni_*$ be a pointed discrete condensed anima.
For a cofiltered limit representation $K=\lim_{i \in I} K_i \in \Extr$ by finite sets $K_i\in \Setfin$, we have
\begin{align*}
\pi_n(A(K), a(K))&=\pi_n(\colim_{i \in I^{\mathrm{op}}}A^{K_i}, a(K))=\colim_{i \in I^{\mathrm{op}}} \prod_{k\in K_i}\pi_n(A, a(\ast))\\&=\colim_{i \in I^{\mathrm{op}}} \Hom_{\Set}(K_i,\pi_n(A, a(\ast)))=\pi_n(A, a(\ast))^{\mathrm{disc}}(K)
\end{align*}
where $(\cdot)^{\mathrm{disc}}\colon\Grp \hookrightarrow \CondGrp$. We used that simplicial homotopy group functors commute with finite products and filtered colimits. The above means that for all discrete condensed anima, the condensed homotopy are discrete and can be identified with the classical simplicial homotopy groups of the corresponding anima. \end{example}
The example applies to the singular simplicial complex $\Sing(T)\in \Ani$ of a topological space $T$. However, viewing a topological space as a condensed anima via the topological direction, gives a trivial result on condensed homotopy groups.
\begin{remark}{\textbf{Vanishing condensed homotopy groups}}\\
As the category $\CondSet$ of condensed sets forms the sub-$\infty$-category of $0$-truncated objects inside $\CondAni$, 
the condensed homotopy groups vanish for all these. In other words, condensed homotopy groups loose all the information on homotopy groups of topological spaces.
\end{remark}
The remark provides one reason to study the notion of shape for condensed anima and its relation to the shape of topological spaces. In the main part, we will construct and study a 'shape' functor
$$\Shape^c\colon \CondAni \rightarrow \ProAni.$$
Our results there compared with the previous remark then imply that
the following square does not commute for $n\geq1$:
\[
\begin{tikzcd}
 \CondAni_*\ar[r]\ar[d, "\pi_{n}^{\cond}"] & \Pro(\Ani)_* \ar[d, "\Pro(\pi_{n})"] \\
 \CondGrp \ar[r] & \Pro(\Grp).
\end{tikzcd}
\]
This is illustrated by the following simple example:
\begin{example}
Consider the static condensed anima corresponding to the circle $S^1$. Its image under $\CondAni \rightarrow \Pro(\Ani)$ then coincides with its homotopy type $\Sing(S^1)$ and its homotopy pro-groups thus with its classical homotopy groups, see \cref{thm:masterthesis}. By truncatedness, all the condensed homotopy groups of $S^1\in \CondAni$ are trivial but the fundamental group $\pi_1(S^1)=\ZZ$ is not. For $n>1$, we can argue in a similar way by choosing an arbitrary CW-space admitting non-trivial higher homotopy groups.
\end{example}
As an upshot, condensed homotopy groups assign a notion of homotopy groups carrying the additional structure of a topology in a broader sense to a condensed anima. However, this is not very meaningful in the context of viewing condensed sets as a replacement of topological spaces. 
\section{Shapes of condensed anima and topological spaces}
How to possibly do homotopy theory for condensed anima, has already been studied in \cite{arXiv:2105.07888}.
There we asked to what extent the discrete functor $$\Ani\hookrightarrow \CondAni$$ has a left adjoint assigning a homotopy type to every condensed anima.
This was motivated by the fact that a hypothetical left adjoint would reflect the situation of the topological direction, where we have the left adjoint $$\pizerocond\colon \CondAni \rightarrow \CondSet$$ by which a condensed anima can be provided with an underlying topological space. In the next subsections, we will build a bridge to the results in \cite{arXiv:2105.07888} and give new perspectives on the notion of homotopy types or, how they will be called from now on, \emph{shapes} for condensed anima.
\subsection{Definition of shapes of condensed anima}\label[subsection]{defshapecondani}
Via the restricted shape functor $\Shape \colon \Top \rightarrow \ProAni$ on topological spaces, we can assign a shape $\Shape(K)\in \ProAni$ to every extremally disconnected profinite set $K\in\Extr$.
As these spaces constitute the building blocks of the $\infty$-topos $\CondAni$ in the sense that $\CondAni$ is freely generated under sifted colimits by $\Extr$, we can extend this to all condensed anima.
\begin{definition}{\textbf{Shapes of condensed anima}}\label[definition]{def:c_shape}\\
We define the \textit{shape functor} on condensed anima
$$\Shape^c\colon \CondAni \rightarrow \ProAni$$ 
as the unique hypercomplete cosheaf extending the restricted shape functor on topological spaces $\Shape \colon \Extr \rightarrow \ProAni$.
For every $X\in \CondAni$, we refer to the image $$\Shape^c(X)\in \ProAni$$ as its \textit{shape}.
\end{definition}
Being a hypercomplete cosheaf is the same as being a functor sending colimits in the hypercomplete $\infty$-topos $\CondAni$ to colimits in $\ProAni$.
As every condensed anima can be represented by a sifted colimit of objects in $\Extr\subset \CondAni$, the construction in the definition of the shape is the extension of the restricted shape functor $\Shape \colon \Extr\rightarrow \ProAni$ by sifted colimits along the Yoneda embedding $\Extr \hookrightarrow \CondAni$.
\begin{remark}\label[remark]{rem:cshape_extdis}
Let $X\in \CondAni$ be represented by a sifted colimit $X=\colim_{i\in I} K_i$ of (extremally disconnected) profinite sets $K_i\in \Extr$.
Then the shape of $X$ is given by
$$ \Shape^c(X)= \colim_{i \in I} \Shape(K_i) \in \ProAni\period $$
By definition, it is clear that for every $K\in \Extr$, viewed as a condensed anima, the shape $\Shape^c(K)$ coincides with the shape $\Shape(K)$.
As mentioned in \cref{shape_profinset}, this shape can be identified with the pro-anima corresponding to the profinite set $K$ itself under the inclusion $\ProFin \subset \ProAni$.
Thus, by abuse of notation, we can simplify to
$$ \Shape^c(X)=\colim_{i \in I} K_i \in \ProAni\period $$
Note that the colimit is taken in the $\infty$-category $\ProAni$ and thus corresponds to a limit of finite limit preserving functors $\Ani \rightarrow \Ani$. The shape of $X\in \CondAni$ is independent of the choice of the colimit realizing $X$.
\end{remark}
We are going to explore topological spaces $T$ for which the shapes $\Shape^c(T)$ and $\Shape(T)$ can be identified. 
\subsection{The shape functor as a left adjoint}\label[subsection]{Shapeandprohomo}
The discrete functor $(\cdot)^{\mathrm{disc}}\colon\Ani \hookrightarrow \CondAni$ preserves finite but not all limits.
Hence, considerations on a left adjoint to the discrete functor given on all condensed anima are only wishful thinking. Nevertheless, we show in \cite[Proposition 7.2.2]{arXiv:2105.07888} that one can fix this by forcing the discrete functor to commute with all limits by taking right Kan extension along $\Ani \hookrightarrow \ProAni$.
Note that this is exactly what happens in the construction of the pro-left adjoint in \cref{proleftadjoint}. One obtains commutative triangles
\begin{equation}\label{triangle-prohomotopy}
\begin{tikzcd}[row sep=huge]
&
\Ani \arrow[hookrightarrow,dl,swap, "(\cdot)^{\mathrm{disc}}"] \arrow[hookrightarrow,dr, "j_{\Ani}"] & 
\\ 
\CondAni \arrow[rr, swap,  yshift=-0.7ex, "L_{\Ani}"] & & \Pro(\Ani) \period \arrow[ll, swap, yshift=+0.7ex, "R_{\Ani}"]
\end{tikzcd}
\end{equation}
The functor $R_{\Ani}$ is the unique cofiltered limit preserving extension of the discrete functor $(\cdot)^{\mathrm{disc}}\colon \Ani \hookrightarrow \CondAni$, $L_{\Ani}$ is its left adjoint.
Commutativity of the triangle involving $L_{\Ani}$ is a consequence of $(\cdot)^{\mathrm{disc}}$ being fully faithful.
The existence of $L_{\Ani}$ enables us to assign a pro-anima to every condensed anima, a cofiltered system of homotopy types, so to say, a \textit{pro-homotopy type}.
The main results of \cite[Section 7.1]{arXiv:2105.07888} are summarized in the following theorem.
\begin{theorem}{\cite[Section 7.1]{arXiv:2105.07888}}\label[theorem]{thm:masterthesis}
The unique cofiltered limit preserving extension of the discrete functor $(\cdot)^{\mathrm{disc}}\colon \Ani \hookrightarrow \CondAni$, denoted by $R_{\Ani}\colon\Pro(\Ani)\rightarrow \CondAni$, admits a left adjoint $L_{\Ani}\colon \CondAni \rightarrow \Pro(\Ani)$
such that 
\begin{enumerate}
    \item[(i)] its restriction to $\Ani$ is the identity,
    \item[(ii)] on profinite sets it agrees with the inclusion $\ProFin \subset \ProAni$ and
    \item[(iii)] on CW-spaces it recovers the classical homotopy type $\Sing\colon \CW \rightarrow \Ani$.
    \end{enumerate} Moreover, extending the simplicial homotopy group functors on (pointed) anima
\begin{align*}
&\pi_0\colon \Ani \ \rightarrow \Set\\
&\pi_1\colon \Ani_{\ast} \rightarrow \Grp\\
&\pi_n \colon \Ani_{\ast} \rightarrow \Ab\Grp, \text{ for $n\geq 2$ }
\end{align*}
to unique cofiltered limit preserving functors on the corresponding pro-$\infty$-categories 
\begin{align*}
&\Pro(\pi_0)\colon \Pro(\Ani) \ \rightarrow \Pro(\Set)\\
&\Pro(\pi_1)\colon \Pro(\Ani)_{\ast} \rightarrow \Pro(\Grp)\\
&\Pro(\pi_n) \colon \Pro(\Ani)_{\ast} \rightarrow \Pro(\Ab\Grp)
\end{align*}
yields, by composition with the left adjoint $\CondAni \rightarrow \Pro(\Ani)$, a notion of homotopy pro-groups on condensed anima
\begin{align*}
&\widetilde{\pi}_0\colon \CondAni \ \rightarrow \Pro(\Set)\\
&\widetilde{\pi}_1\colon \CondAni_{\ast} \rightarrow \Pro(\Grp)\\
&\widetilde{\pi}_n \colon \CondAni_{\ast} \rightarrow \Pro(\Ab\Grp).
\end{align*}
Then the homotopy pro-groups
\begin{itemize}
    \item[(i)*] agree with the simplicial homotopy groups on all $\Ani \subset \CondAni$,
    \item[(ii)*] are trivial on $\ProFin\subset \CondAni$ for all $n\geq 1$,
    \item[(iii)*] recover the usual homotopy groups from topology for all $\CW \subset \CondAni.$
\end{itemize}
\end{theorem}
The shape functor on condensed anima as defined in \cref{def:c_shape} gives a more explicit description of the left adjoint $L_{\Ani}$ constructed and studied in \cite[Proposition 7.2.2]{arXiv:2105.07888}. Note that there is an error in the description of the functor $R_{\Ani}$ in \cite[Lemma 7.2.1]{arXiv:2105.07888} due to non-commutativity of limits and colimits.
\begin{corollary}\label[corollary]{shape_prohomotopy_agree}
The shape functor $\Shape^c$ is left adjoint to the unique cofiltered limit preserving pro-extension
$$R_{\Ani}\colon \ProAni \rightarrow \CondAni, \ P=\{A_i\}_{i\in I} \mapsto \lim_{i \in I}\ A_i^{\mathrm{disc}}$$
of the discrete functor $(\cdot)^{\mathrm{disc}}\colon \Ani \hookrightarrow \CondAni$. For every $X \in \CondAni$, the image is given by the pro-anima
$$\Ani \rightarrow \Ani, \ A \mapsto \Hom_{\CondAni}(X, A)\period $$
\end{corollary}
\begin{proof}
For the left adjointness, 
we show that there is a natural isomorphism in $X\in \CondAni$ and $P\in \ProAni$
$$\Hom_{\CondAni}(X,R_{\Ani}(P))\simeq\Hom_{\ProAni}(\Shape^c(X),P)\period$$
By the constructions of both $\Shape^c$ and $R_{\Ani}$, we are reduced to $X=K\in \Extr \subset \CondAni$ and $P=A\in \Ani \subset \ProAni$. For a representation $K=\{K_s\}_{s\in S}$ by finite sets $K_s\in \Setfin$, we deduce
\begin{align*}
\Hom_{\CondAni}(K,R_{\Ani}(A))&=\Hom_{\CondAni}(K,A^{\mathrm{disc}})\\&\simeq\colim_{s\in S^{\op}}A^{K_s}\\&\simeq\colim_{s\in S^{\op}}\Hom_{\Ani}(K_s,A)\\&\simeq\Hom_{\ProAni}(K,A)\period
\end{align*}
Here the second line follows by the Yoneda lemma and the definition of the discrete functor, see \cref{rec:discrete_functor}, and the last conclusion is by the definition of morphisms in pro-categories. Altogether, the shape functor coincides with the pro-left adjoint of the discrete functor. Thus, the description follows by the general formula for pro-left adjoints in \cref{proleftadjoint}.
\end{proof}
The existence of the right adjoint $R_{\Ani}$ has formal consequences which we explain further below. 
\cref{thm:masterthesis} above explains how to define homotopy pro-group functors on condensed anima. 
In \cref{sec:homotopyinv}, we formulate a more general version of (iii) for all locally contractible spaces. 
\begin{remark}
For those topological spaces $T\in \CondAni$ with $\Shape^c(T)=\Sing(T),$ e.g., all locally contractible spaces as we will see in \cref{Locally_contractible_cshape}, the homotopy pro-groups are the usual homotopy groups $\pi_n(T,x)$ of $T$ as a topological space as these agree with the simplicial homotopy groups of $\Sing(T) \in \Ani$, see \cite[\href{https://Kerodon.net/tag/00VR}{Example 00VR}]{Kerodon}.
\end{remark}
\begin{remark} The proof of \cite[Proposition 7.2.9]{arXiv:2105.07888} claiming (iii) and (iii)* uses \cite[Lemma 11.9]{Scholze:analyticnotes} by Clausen-Scholze: 
For every CW-space $T$ (more generally, for every condensed set that is a filtered colimit of condensed sets that are built out of pushouts of $S_{n-1} = \delta D_n \hookrightarrow D_n$), there is a universal anima $Y\in \Ani$ with a map $T\rightarrow Y$ of condensed anima and such that $Y$ is equivalent to the fundamental $\infty$-groupoid $\mathrm{Sing}(T)$. In other words, $\Sing(T)$ constitutes an initial anima with a map $$T\rightarrow \mathrm{Sing}(T)\in \CondAni\period $$
Taking a closer look at the proof of \cite[Lemma 11.9]{Scholze:analyticnotes}, makes clear that it is actually just based on a by hand computation of the assignment $$A\mapsto \Hom_{\CondAni}(T,A)$$ defining the shape.
\end{remark}
Indeed, the existence of an initial anima $Y$ with a map $X\rightarrow Y$ is equivalent to the fact that the shape $\Shape^c(X)$ is given by $Y \in \Ani$. This is based on the shape functor $$\Shape^c\colon \CondAni \rightarrow \ProAni$$ being part of an adjunction.
\begin{lemma}
For every condensed anima $X\in \CondAni$, the shape $\Shape^c(X)$ is the universal pro-anima with a map from $X$ (to its image under the right adjoint $R_{\Ani}$)
$$ X \rightarrow R_{\Ani}(\Shape^c(X)) \in \CondAni\period $$ Further, the shape $\Shape^c(X)\in \ProAni$ lies in $\Ani$ if and only if there exists a universal anima $A\in Ani$ with a map $X \rightarrow A$. Then one has $\Shape^c(X)=A$.
\end{lemma}
\begin{proof}
 We can map the shape given by $$\Shape^c(X)=L_{\Ani}(X)\in \Pro(\Ani)$$ under the right adjoint $$R_{\Ani}\colon \Pro(\Ani) \rightarrow \CondAni$$ back to an object $D\in\CondAni$ and obtain a map $X\rightarrow D$ by the unit of the adjunction. This map is universal in the sense above just by formal properties of the adjunction. As every anima is in particular a pro-anima, we conclude the corresponding statement for $\Shape^c(X)$ lying in the full subcategory $\Ani \subset \ProAni$.
 \end{proof}
\begin{corollary}
For every condensed anima $X=\colim_{i \in I} X_i$ being a colimit of condensed anima $X_i$ whose shape $\Shape^c(X_i)$ lies in $\Ani$, there exists a universal anima $A$ with a map $X\rightarrow A$ and coinciding with $\Shape^c(X)$. 
\end{corollary}
\begin{proof}
As colimits are preserved under the shape functor, the shape of $X$ is given by 
$$\Shape^c(X)=\colim_{i \in I} \Shape^c(X_i) \in \ProAni\comma $$
where by assumption $\Shape^c(X_i) \in \Ani$. As the Yoneda embedding $j_{\Ani}\colon \Ani \hookrightarrow \ProAni$ is a colimit preserving left adjoint, $\Shape^c(X)$  does also correspond to an object in $\Ani$. The rest follows from the last lemma.  
\end{proof}
The corollary for example applies to locally contractible spaces. For such spaces $T$, the universal anima $\Shape^c(T)$ is given by $\Sing(T)$, see \cref{sec:homotopyinv}.
\subsection{Shapes of slices and the \'etale \texorpdfstring{$\infty$}{infinity}-topos}\label[subsection]{connectionshape}
In the preceding sections, we have seen two different points of view on the shape of condensed anima.
Here we will study the situation from a more distant point of view of shapes of $\infty$-topoi: The shape functor on all topological spaces is defined as the composite 
$$\Shape\colon \Top \xrightarrow{\Sh(\cdot)} \RTop_{\infty}\xrightarrow{\Shape} \ProAni\period $$ 
Similarly, by the subsequent observations, the shape functor $\Shape^c$ coincides both with the composite
$$\CondAni \xrightarrow{\CondAni_{/\cdot}} \RTop_{\infty}\xlongrightarrow{\quad \Shape \quad} \ProAni\comma $$
where $\CondAni_{/\cdot}$ assigns the  slice $\infty$-topos $\CondAni_{/X}$ to a condensed anima $X$, as well as the composite
$$\CondAni \xrightarrow{\Sh_{\et}(\cdot) } \RTop_{\infty}\xlongrightarrow{\quad \Shape \quad} \ProAni\comma $$
where the \textit{\'etale $\infty$-topos} $$\Sh_{\et}(X)\simeq \FunctsInt(X, \Ani^{\mathrm{ult}})$$ of a condensed anima $X$ will come up as a full subcategory of
$$\CondAni_{/X}\simeq \FunctsInt(X, \ICond(\Ani))\period $$
This is opening up a new approach to compare the shape $\Shape$ and the shape $\Shape^c$ for arbitrary topological spaces by studying the relations of $\CondAni_{/X}$, $\Sh_{\mathrm{\et}}(X)$ and $\Sh(X)$. 
\subsubsection{Obtaining the shape via shapes of slices and the \'etale \texorpdfstring{$\infty$}{infinity}-topos}
For every condensed anima $X\in \CondAni$, we can consider the slice $\CondAni_{/X}$, which comes with an \textit{essential geometric morphism} $F\colon \CondAni_{/X} \rightarrow \CondAni$, i.e., an adjoint triple
\[
 \begin{tikzcd}[column sep = huge]
  \CondAni_{/X}   \arrow[r, shift left=2.5ex, "f_*"] \arrow[r, shift left=-2.5ex, "f_!"] &\CondAni. \arrow[l, shift left=0ex, "f^*", swap]
          \end{tikzcd}
    \]
The left exact left adjoint $f^*$ is the base change induced by the terminal morphism $f\colon X\rightarrow \ast$, i.e., it is given by pullback along $f$. Its left adjoint $f_!$ is given by postcomposition with $f$.
\begin{corollary}\label[corollary]{relativeshape}
The shape $\Shape^c\colon \CondAni \rightarrow \ProAni$ computes the shape $\Pi_{\infty}(\CondAni_{/\cdot})$.
In other words, it coincides with the composed functor $$\CondAni \xrightarrow{\CondAni_{/\cdot}} \RTop_{\infty}\xlongrightarrow{\quad \Shape \quad} \ProAni\period $$
\end{corollary}
\begin{proof}
The functor $\Shape^c=L_{\Ani}\colon \CondAni \rightarrow \Pro(\Ani)$ is unique in the sense of being a left adjoint of the unique cofiltered limit preserving extension of $\Ani \hookrightarrow \CondAni$. We have a commutative triangle 
\[
\begin{tikzcd}[row sep=huge]
&
\CondAni_{/X}\arrow[rightarrow,dl, "f_!"] \arrow[rightarrow,dr, "h_!"] & 
\\ 
\CondAni  \arrow[rr,  "L_{\Ani}"] & & \Pro(\Ani),
\end{tikzcd}
\]
where $h_!$ is the pro-left adjoint of $h^*\colon \Ani \rightarrow \CondAni_{/X}$. As $X\rightarrow X$ is the terminal object in $\CondAni_{/X}$ and its image under $f_!$ is $X\in \CondAni$, we deduce the claim under notice of $\Pi_{\infty}(\CondAni_{/X})=h_!(X)$.
\end{proof}
This alternative description of the shape helps to find its relation to the classical shape in the context of topological spaces: For every topological space $T$, there is a comparison geometric morphism
\begin{align}\label{compmor}
c_T\colon\CondAni_{/T} \longrightarrow \Sh(T)
\end{align}
inducing a comparison map on pro-anima $$\Shape^c(T) \rightarrow  \Shape(T)\in \ProAni\period $$ 
\begin{remark}{\cite[Observation 3.9]{arXiv:2210.00186}}
The pushforward $c_{T,*}$ of (\ref{compmor}) can be described explicitly by the formula $$c_{T,*}(G)(U)=\Hom_{\CondAni_{/T}}(U, G)\comma $$
where the open subset $U\subset T$ is also viewed as an object in $\CondAni_{/T}$. In the reversed direction, the pullback morphism $c^*_T$ is the colimit preserving extension of sending the representable sheaf $h_U \in \Sh(T)$ corresponding to an open subset $U\subset T$ to the static condensed anima assigned to $U\in \Top$ via $\Top \rightarrow \CondAni$.
\end{remark}
We study the comparison morphism (\ref{compmor}) in order to connect the shape $\Shape(T)$ of $T\in \Top$ with the shape $\Shape^c(T)$ of $T\in \CondAni$. 
\begin{example}
In \cref{Locally_contractible_cshape}, we will show that the shape $\Shape^c(T)$ of a locally contractible space $T$ identifies with $\Sing(T)$. In the proof, we construct for such a space $T$ a hypercovering of disjoint unions of contractible opens such that $T$ is the colimit of the hypercovering
in $\CondAni$ and in $\Shhyp(T)$.
By the definition of the colimit preserving pullback 
$$c^{*}_T\colon\Sh(T) \rightarrow \CondAni_{/T}$$
and since it factors over the hypercompletion
$$c^{*, \mathrm{hyp}}_T\colon \Shhyp(T) \rightarrow \CondAni_{/T}\comma $$
by $\CondAni_{/T}$ being hypercomplete, this colimit representation of $T$ is preserved under 
$c^{*, \mathrm{hyp}}_T$. 
As we will further see that the shape functor $\Shape^c(\cdot)$ is homotopy invariant (\cref{condensedhomotopyinv}), and thus $*=\Shape(C)=\Shape^c(C)$ for every contractible space $C$, this shows that the comparison map $$\Shape^c(T) \rightarrow \Shape(\Sh^{\mathrm{hyp}}(T))=\Sing(T)$$ is an equivalence for all locally contractible spaces, see \cref{locallycontrshape}.
\end{example}
We want to discuss the situation a little more broadly:
There is a general sufficient condition on the comparison morphism to induce an equivalence on shapes.
Namely by \cref{fullyshapeequiv}, this is the case whenever the pullback $$c_{T}^*\colon \Sh(T)\longrightarrow \CondAni_{/T}$$ is fully faithful.
This was examined in more detail by Haine.
\begin{example}{\cite[Corollary 4.7]{arXiv:2210.00186}}
For $K\in \Extr$, the pullback $c_K^*$ always is fully faithful. 
\end{example}
By the universal properties of Postnikov completion and hypercompletion and $\CondAni_{/T}$ being Postnikov complete, the right adjoint pushforward $$c_{T,*}\colon \CondAni_{/T} \longrightarrow \Sh(T)$$ always factors through
$$\Sh^{\mathrm{post}}(T)\rightarrow \Shhyp(T) \rightarrow \Sh(T)\period $$
In \cite[Lemma 4.15., Lemma 4.17. and Lemma 4.18]{arXiv:2210.00186}, Haine points out that hypercompletion and convergence of Postnikov towers in $\Shhyp(T)$ is crucial for $c_T^*$ being fully faithful.
Even for arbitrary compact Hausdorff spaces, we do not necessarily get a fully faithful pullback \cite[Section 4.3]{arXiv:2210.00186}.
Nevertheless, we have the following:
\begin{corollary}\label[corollary]{prohoshape}
Let $T$ be a locally compact Hausdorff space. Then the Postnikov completed comparison morphism $c_{T}^{*,\mathrm{post}}$ induces an equivalence of shapes
$$\Shape^c(T)\simeq \Pi_{\infty}^{\mathrm{post}}(T) $$ which gives
on pro-truncated shapes
$$\Pi_{<\infty}^c(T)\simeq\Pi_{<\infty}(T)\period$$
\end{corollary}
\begin{proof}
By \cite[Corollary 4.9. and Corollary 4.11]{arXiv:2210.00186}, for every locally compact Hausdorff space $T$, the comparison morphism (\ref{compmor}) induces a fully faithful functor on Postnikov completions 
$$c_{T}^{*,\mathrm{post}}\colon \Sh^{\mathrm{post}}(T)\hookrightarrow \CondAni_{/T}\period $$
By \cref{fullyshapeequiv} and \cref{relativeshape}, we have an equivalence of the shapes $\Pi_{\infty}^{\mathrm{post}}(T)\colonequals \Shape(\Sh^{\mathrm{post}}(T))$ and $\Shape^c(T)$.
On the other hand, by \cref{protruncation_shapeequiv} the shape $\Pi_{\infty}^{\mathrm{post}}(T)$ is, up to pro-truncation, equivalent to the shape $\Pi_{\infty}(T)$.
\end{proof}
As we cannot expect the $\infty$-topos of a general topological space to be determined by the $\infty$-topoi of all (extremally disconnected) compact Hausdorff spaces mapping into that space, we make the following definition suggested by Clausen. See the online discussion at \cite{Clausen2020Xena} and further at \cite{Ramzi2021}.
\begin{definition}\label{def:etalesheaves}
Let $X\in \CondAni$. Then $X=\colim_{K\rightarrow X}K$, where the colimit runs over all extremally disconnected profinite sets $K\in \Extr$ mapping to $X$. We define the \textit{\'{e}tale} $\infty$-\textit{topos} of $X$ as $$\Sh_{\mathrm{\et}} (X)\colonequals\lim_{K\rightarrow X} \Sh(K)\period $$
\end{definition}
Fully faithfulness of the comparison morphisms
$$c_K^* \colon \Sh(K) \hookrightarrow \CondAni_{/K}$$ for all $K\in \Extr$ then gives, by taking the limit, fully faithfulness of the morphism
\begin{align}\label{ffovertop}
\Sh_{\mathrm{\et}}(X)=\lim_{K\rightarrow X} \Sh(K)\longrightarrow \lim_{K\rightarrow X} \CondAni_{/K}=\CondAni_{/X} \period
\end{align}
Note that the identification on the right hand side holds as all colimits in the $\infty$-topos $\CondAni$ are van Kampen. The morphism (\ref{ffovertop}) induces an equivalence $$\Pi_{\infty}(\Sh_{\mathrm{\et}}(X))\simeq\Shape^c(X)\in \ProAni\period $$ 
\begin{remark}
By fully faithfulness of (\ref{ffovertop}) and \cite[Lemma 4.18]{arXiv:2210.00186}, all Postnikov towers in $\Sh_{\mathrm{\et}}(X)$ converge. Indeed, the \'etale $\infty$-topos is Postnikov complete as all $\Sh(K)$ are.
\end{remark}
The \'etale $\infty$-topos of a condensed anima defines by left Kan extension along 
$$\Extr \rightarrow \RTop_{\infty}, \quad K \mapsto \Sh(K)\period $$
a colimit preserving functor
$$\CondAni \xrightarrow{\Sh_{\mathrm{\et}}(\cdot)} \RTop_{\infty}, \quad X \mapsto\Sh_{\mathrm{\et}}(X)\period $$ Then the following corollary is a direct consequence of the statements above.
\begin{corollary}\label[corollary]{cor:ettopos_cshape}
The shape functor
$$\Shape^c\colon \CondAni \rightarrow \ProAni$$ agrees with the composite
$$\CondAni \xrightarrow{\Sh_{\mathrm{\et}}} \RTop_{\infty}\xrightarrow{\Shape} \ProAni\period $$
\end{corollary}
Thus, in order to compare the shape $\Shape(T)$ of a topological space $T\in \Top$ with its shape $\Shape^c(T)$ as a condensed anima $T\in \CondAni$, we can study almost just as well the relation of $\Sh(T)$ and the \'{e}tale $\infty$-topos $\Sh_{\mathrm{\et}}(T)$. 
\begin{example}
For every topological space $T$, there is a pullback morphism 
\begin{align}\label{mor:pullback_etale}
\Sh(T) \rightarrow\Sh_{\mathrm{\et}}(T)    
\end{align}
induced by the pullback morphisms over all $K\rightarrow T$, with $K\in \Extr$\comma $$\Sh(T) \rightarrow \Sh(K)\period $$
If $T$ is a compact Hausdorff space, Haine's results in \cite[Corollary 2.8]{arXiv:2210.00186} imply that the pullback (\ref{mor:pullback_etale}) is the Postnikov completion of $\Sh(T)$.   
\end{example}
In Lurie's work on ultracategories, he proves that left ultrafunctors $\Fun^{\mathrm{LUlt}}(T, \Set)$ from a compact Hausdorff space $T$ to the ultracategory $\Set$ recover the sheaf topos $\Sh(T, \Set)$, see \cite[Theorem 3.4.4]{Ultracategories}. Here, we have a variant of this result for all condensed anima and their \'etale $\infty$-topoi by equipping $\Ani$ with a condensed structure.
\begin{definition}
We define a condensed $\infty$-category $\Ani^{\mathrm{ult}}$ by the assignment
\begin{align*}
&\Extr^{\op}\rightarrow \Catinfty, \ K\mapsto \Sh(K)
\end{align*}
which we refer to as the \textit{condensed $\infty$-category of anima}.    
\end{definition}
The superscript 'ult' refers to the fact that the condensed structure on $\Ani$ arises via the natural ultrastructure on $\Ani$.
\begin{proposition}\label[proposition]{prop:localisos_etale}
For every $X\in \CondAni$, we have an equivalence
$$\Sh_{\mathrm{\et}}(X)\simeq\FunctsInt(X, \Ani^{\mathrm{ult}})\comma $$ where we view $X$ as a condensed $\infty$-category via $\CondAni \hookrightarrow \CondCat$.
\end{proposition}
\begin{proof}
We show the claim for $K\in \Extr$ and then argue by a limit argument: For $K \in \Extr$, the identification
$$\Sh_{\mathrm{\et}}(K)=\Sh(K)=\FunctsInt(K, \Ani^{\mathrm{ult}})$$ follows from the $\infty$-categorical Yoneda Lemma. Indeed, $K\in \CondCat$ can be identified with the image of $K$ under the Yoneda embedding $$\iota \colon \Extr \hookrightarrow \CondCat\period $$ Thus, $\FunctsInt(K, \Ani^{\mathrm{ult}})$ is nothing else than the evaluation of $\Ani^{\mathrm{ult}}$ in $K$ which is, just by definition, exactly given by $\Sh(K)$. As natural transformations are homomorphisms of functors, these are compatible with colimits in the first component: For a general $X\in \CondAni$ given by the colimit $X=\colim_{K \rightarrow X} K$, one has $$ \FunctsInt(X, \Ani^{\mathrm{ult}})= \lim_{K \rightarrow X} \FunctsInt(K, \Ani^{\mathrm{ult}})$$ 
and hence 
$$\Sh_{\mathrm{\et}}(X)=\lim_{K \rightarrow X} \Sh(K)=\lim_{K \rightarrow X} \FunctsInt(K, \Ani^{\mathrm{ult}})=\FunctsInt(X, \Ani^{\mathrm{ult}})$$
finishing the proof.
\end{proof}
\begin{remark}
In particular, the fully faithful pullback morphism 
$$\Sh_{\mathrm{\et}}(X)\hookrightarrow \CondAni_{/X}$$ corresponds to the canonical fully faithful pullback morphism
$$\FunctsInt(X, \Ani^{\mathrm{ult}})\hookrightarrow \FunctsInt(X, \ICond(\Ani))\comma$$
which is induced by the fully faithful embedding $\Ani^{\mathrm{ult}} \hookrightarrow \ICond(\Ani)$ of condensed $\infty$-categories.
Note that the identification
$$\CondAni_{/X}\simeq\FunctsInt(X, \ICond(\Ani))\comma $$    
is a condensed version of the Grothendieck construction proven in \cite[Corollary 3.20]{MR4574234}.
\end{remark}
\begin{example}
For every anima $A\in \Ani \subset \CondAni$, we have an equivalence
$$\Sh_{\et}(A)\simeq\Ani_{/A}\period $$
Indeed, there is an equivalence of $\infty$-categories
$\Functs(A, \Ani^{\mathrm{ult}})=\Fun_{\Catinfty}(A, \Ani)$:\\ Every anima is given as a sifted colimit of finite discrete sets $A=\colim_{i \in I} A_i$ and we have
\begin{align*}
\Functs(A, \Ani^{\mathrm{ult}})
&=\lim_{i \in I}\Functs(A_i, \Ani^{\mathrm{ult}})\\&=\lim_{i \in I}\Sh(A_i)=\lim_{i \in I} \prod_{a\in A_i} \Sh(\ast)\\&=\lim_{i \in I} \prod_{a\in A_i} \Ani=\lim_{i \in I}\Ani_{/{A_i}}\\&=\Ani_{/A}=\Fun_{\Catinfty}(A, \Ani)
\end{align*}
where we used the $\infty$-categorical Yoneda lemma and the Grothendieck construction.
\end{example}
\subsection{Homotopy invariance of the shape functor}\label{sec:homotopyinv}
Here we prove that the shape functor on condensed anima likewise satisfies homotopy invariance. This can be used to determine $\Shape^c(T)$ for a larger class of topological spaces, namely all locally contractible spaces $T$. To do so, we recall properties of paracompact spaces.
\begin{recollection}\label[remark]{rem:comp_para}
We call a topological space $T$ \textit{paracompact} if it is Hausdorff and every open cover can be refined by a locally finite cover, i.e., a cover $\{U_i \subset T\}_{i \in I}$ of open subsets such that for every $x\in T$, there exists some neighborhood $x\in U_x$ having non-empty intersection with only finitely many of the $U_i$.
This class of spaces contains all CW-complexes, compact Hausdorff spaces and second countable locally compact Hausdorff spaces, for references see \cite[Proposition 1.18, Proposition 1.20]{Hatcher2017para} and \cite[Theorem 2.6]{Conrad}.    
\end{recollection}
For a paracompact space, there is defined the notion of \emph{covering dimension} as follows.
\begin{definition}
A paracompact topological space $T$ has \emph{covering dimension} $\leq n$ if for any open covering
$\{U_i\}_{i \in I}$ of T, there exists an open refinement $\{V_j\}_{j\in J}$ of $T$ such that each intersection $V_{j_0}\cap\cdots \cap V_{j_{n+1}}$, with pairwise distinct $j_k$, 
is empty.  
\end{definition}
\begin{recollection}\label{rec:homotopydimension}
In \cite[Theorem 7.2.3.6]{HTT}, Lurie shows that the sheaf $\infty$-topos of a paracompact space with finite covering dimension has finite \emph{homotopy dimension} \cite[Definition 7.2.1.1]{HTT}. Further in \cite[Proposition 7.2.1.10]{HTT}, he shows that an $\infty$-topos of locally finite homotopy dimension is Postnikov complete. Here, the local condition means that there exists a family $\{U_{i}\}$ of objects of the $\infty$-topos $\mathcal{X}$ which generate
$\mathcal{X}$ under colimits, such that each $\mathcal{X}_{/U_i}$
is of homotopy dimension $\leq n$.
\end{recollection}
We give examples of spaces with  finite covering dimension and argue that the sheaf $\infty$-topos on $K\times [0,1]$ is Postnikov complete for every profinite set $K$.
\begin{example}\label{ex:profinitepostnikov}
By \cite[Remark 7.2.3.3]{HTT}, every profinite set has covering dimension $\leq 0$, as it is a cofiltered limit of finite discrete spaces all of which have covering dimension $0$. It has a basis by clopen subsets, which are again profinite, and thus have covering dimension $\leq 0$ as well.
\end{example}
\begin{example}
Every interval (no matter if closed, open or half open) is paracompact and has covering dimension $\leq 1$. This standard result can, e.g., be deduced from combining \cite[Theorem 1.8.2, Corollary 1.8.3, Theorem 3.1.4]{Engelking1978} and the fact that homeomorphic spaces have the same covering dimension. 
\end{example}
\begin{lemma}\label{lem:productPostnikov}
For every profinite set $K$, the sheaf $\infty$-topos $\Sh(K\times [0,1])$ is Postnikov complete.
\end{lemma}
\begin{proof}
By \cref{rec:homotopydimension}, it is enough to argue that the product space $K\times [0,1]$ has a basis by paracompact spaces of finite covering dimension $\leq n$ for some $n$. We choose a basis of $K\times [0,1]$ by spaces $C\times D$, where the $C\subset K$ are clopen subsets and $D$ are intervals open in $[0,1]$. Every space $C\times D$ has finite covering dimension. This follows by a product formula for covering dimensions, e.g., in the version of \cite[Theorem 1]{MORITA1973125}. It states that the covering dimension of $C\times D$ is smaller or equal the sum of the covering dimensions of $C$ and $D$, which is $\leq 1$ by the preceding examples. Thus, the covering dimension of the product space is $\leq 1$. The theorem applies as all $C$ and $D$ are normal, paracompact and locally compact. The $\infty$-topos $\Sh(K\times [0,1])$ is generated by the subspaces $C\times D$ under colimits as they are closed under finite intersections \cite[Lemma C.3]{MR3302973}. Each $C\times D$ is paracompact (e.g., as it is the product of a paracompact space and locally compact paracompact space, see \cite[Theorem 2]{Wenjen1967}). Hence, their corresponding $\infty$-topoi have finite homotopy dimension. This shows that $\Sh(K\times [0,1])$ has locally finite homotopy dimension and thus is Postnikov complete by \cite[Proposition 7.2.1.10]{HTT}, compare \cref{rec:homotopydimension}.
\end{proof}
Before we show that the shape functor on condensed anima satisfies \emph{homotopy invariance}, we first need to define the concept of homotopy equivalences of condensed anima along the lines of topological spaces.
\begin{definition}{\textbf{Homotopy equivalences of condensed anima}}\\
Let $J=\interval{0}{1} \in \CondAni$ be the condensed anima associated with the unit interval and \begin{align*}
*_0\colon *\rightarrow J \text{ and }
*_1\colon *\rightarrow J
\end{align*}
the morphisms of condensed anima corresponding to the points $0, 1\in J=\interval{0}{1}\in \Top$.\\ A map $f\colon T \rightarrow Y$ of condensed anima is a \textit{homotopy equivalence} if there exists a morphism $g\colon Y \rightarrow T$, the \textit{homotopy inverse}, and \textit{homotopies} 
\begin{align*}
H_1\colon T \times J \rightarrow T \text{ and } H_2\colon Y \times J \rightarrow Y
\end{align*}
such that the compositions of $H_1, H_2$ with the maps 
\begin{align*}
T\times * \xrightarrow{\mathrm{id}\times *_i} T \times J \text{ and } Y\times * \xrightarrow{\mathrm{id}\times *_i} Y \times J
\end{align*}
are given by $g\circ f\colon T \rightarrow T$, resp., $f\circ g\colon Y \rightarrow Y$ for $i=0$ and $\mathrm{id}_T$, resp., $\mathrm{id}_Y$ for $i=1$.
\end{definition}
\begin{remark}
As $\Top \rightarrow \CondAni$ is a right adjoint, for every homotopy equivalence of topological spaces $f\colon T \rightarrow Y \in \Top$, the induced map $f \in \CondAni$ is a homotopy equivalence of condensed anima in the sense of the definition.
Moreover, just as in the case of topological spaces, for an $T\in \CondAni$ the projection map $$\mathrm{pr}\colon T \times J \rightarrow T$$ is a common homotopy inverse of the inclusions $i_0, i_1\colon T\simeq T\times \ast \xrightarrow{\mathrm{id}\times \ast_{0,1}} T \times J$. 
\end{remark}
\begin{proposition}\label[proposition]{condensedhomotopyinv}
The shape functor $\Shape^c\colon \CondAni \rightarrow \ProAni$ is homotopy invariant, i.e., inverts homotopy equivalences of condensed anima.
\end{proposition}
\begin{proof}
By the same argument as in \cref{homotopyinv}, it is enough to show that for all $T \in \CondAni$ the shape functor $\Shape^c$ inverts the projection map $T\times [0,1] \rightarrow T$.
So let us show that the induced map $$\Shape^c(T\times \interval{0}{1})\rightarrow \Shape^c(T)$$ is an equivalence in $\ProAni$.
As sifted colimits commute with finite products in every $\infty$-topos, a representation of $T=\colim_{i\in I} K_i$ with $I$ sifted yields a representation $$T\times [0,1]=\colim_{i\in I} (K_i \times [0,1])\period $$
We can reduce to spaces of the form $K \times [0,1]$ with $K$ an extremally disconnected profinite set.
For these spaces, we know by \cref{homotopyinv} that $$\Shape^c(K)=\Shape(K)=\Shape(K\times [0,1])\period $$
So it remains to deduce the equivalence
$$\Shape(K\times [0,1])=\Shape^c(K\times [0,1])\period $$
This follows as $K\times [0,1]$ is a compact Hausdorff space whose sheaf $\infty$-topos is Postnikov complete by \cref{lem:productPostnikov}: According to \cref{prohoshape}, it is $\Shape^c(K\times [0,1])\simeq \Shape^{\mathrm{post}}(K\times [0,1])$ and by Postnikov completeness, we further have $\Shape(K\times[0,1])=\Shape^{\mathrm{post}}(K\times [0,1])$. 
\end{proof}
\begin{remark}
In other words, the shape functor factors as a composition $$\Shape^c\colon \CondAni \rightarrow \CondAni[W^{-1}]\rightarrow \ProAni\comma $$ where $\CondAni[W^{-1}]$ is the localization at the class of homotopy equivalences.    
\end{remark}
In particular, every contractible topological space is also contractible as a condensed anima.
This was also explained by Scholze in \cite{Scholze2022mathflowlocont} in order to conclude the following result. For us a locally contractible topological space has, by definition, a basis of contractible open subsets.
\begin{corollary}\label[corollary]{Locally_contractible_cshape}
For every locally contractible space $T$, we have $\Shape^c(T)\simeq \Sing(T)\in \Pro(\Ani).$
\end{corollary}
\begin{proof}
The proof follows the idea described by Scholze in \cite{Scholze2022mathflowlocont}:
If $T$ is contractible, it is, as a topological space, homotopy equivalent to the one point space.
A homotopy equivalence of topological spaces translates under $\Top \rightarrow \CondAni$ to a homotopy equivalence of condensed anima.
Applying \cref{condensedhomotopyinv}, the shape of $T$ is thus given by $$\Shape^c(T)=\Shape^c(\ast)=\Shape(\ast)=\ast \in \ProAni\period $$ 
Now every locally contractible space $T$ can be covered by contractible opens $T=\cup_{i \in I} U_i$ such that all finite intersections of the $U_i$ again correspond to locally contractible spaces.
Repeating the procedure of covering the intersections by contractible opens, we can gradually build a hypercovering of $T$ with disjoint unions of contractible opens on every level.
Such a hypercovering can be realized both in $\CondAni$ and in the $\infty$-topos $\Shhyp(T)$ of hypersheaves over $T$.
Since the functor $\Top \rightarrow \CondAni$ preserves coproducts, see~\cite[Lemma~3.5]{arXiv:2210.00186}, every disjoint union of contractible spaces in $\CondAni$ belongs to a disjoint union in $\Top$ whereas in $\Shhyp(T)$ they correspond to disjoint unions of representable sheaves corresponding to contractible opens of $T$.
By hypercompleteness of $\CondAni$ and $\Shhyp(T)$, the object $T$ is in both cases equivalent to the colimit of this hypercovering in the respective $\infty$-topos. For every contractible space $C\subset T$, we have $\Shape^{\mathrm{hyp}}(C)=*$ by \cref{locallycontrshape}.
Thus, the images under the colimit preserving (pro-left adjoint) functors $$\Shape^c\colon \CondAni \rightarrow \ProAni \text{ and } g_!\colon \Shhyp(T) \rightarrow \ProAni$$ result in the same colimit of disjoint unions of points which lies in the essential image of $j_{\Ani}\colon \Ani \hookrightarrow \ProAni$.
By \cref{locallycontrshape}, this colimit coincides with $\Sing(T)$ as $g_!(T)=\Shape^{\mathrm{hyp}}(T)=\Sing(T)$.
\end{proof}
By \cref{condensedhomotopyinv}, every condensed anima that is homotopy equivalent to a discrete condensed anima has the shape of this discrete condensed anima.
The shape of such a condensed anima recovers the related anima itself.
\begin{proposition}\label[proposition]{discconshape}
The shape of a discrete condensed anima $A$ is given by the anima 
$$\Shape^c(A)=A \in \Ani\comma $$
regarded as a pro-anima via the Yoneda embedding $j_{\Ani}\colon \Ani \hookrightarrow \ProAni$.
\end{proposition}
\begin{proof}
Every discrete condensed anima $A\in\CondAni$ is a sifted colimit $A=\colim_{i \in I} S_i$ of finite sets $S_i\in \Setfin$ as this is true inside $\Ani$ and $\Ani \hookrightarrow \CondAni$ is a left adjoint.
Therefore, its shape appears as 
$$\Shape^c(A)=\colim_{i \in I} \Shape(S_i) \in \ProAni\period $$ Regarding a finite set as a finite discrete CW-complex, we know via $\Setfin \hookrightarrow \Ani$ that $$\Shape(S_i)=\Sing(S_i)=S_i \in \Ani\period $$ Since $j_{\Ani}\colon \Ani \hookrightarrow \ProAni$ is a left adjoint, $\colim_{i \in I} S_i \in \ProAni$ coincides with the image $j_{\Ani}(A)=j_{\Ani}(\colim_{i \in I} S_i)=\colim_{i \in I} j_{\Ani}(S_i)\in \ProAni.$
\end{proof}
%

\subsection{Comparison results for cohomology theories}\label[subsection]{Cohomological_results}
We have studied how the passage from topological spaces to condensed anima is compatible with the notion of shape on topological spaces. Here, we will summarize how this has consequences on other important invariants of topological spaces to be detected on the associated condensed anima, namely cohomology. 
\begin{recollection}
In the lecture notes on condensed mathematics by Clausen-Scholze \cite{Scholze:condensednotes}, there is a comparison result on condensed and sheaf cohomology of topological spaces based on work by Dyckhoff \cite[Theorem 3.2]{Scholze:condensednotes}. They show that for every compact Hausdorff space $T$ the sheaf cohomology with coefficients in abelian groups agrees with the cohomology internally to the topos of condensed sets, i.e., for every abelian group $A\in \Ab\Grp$ there is an isomorphism    
\begin{align}\label{cohomocomp}
H^*_{\mathrm{sheaf}}(T;A) \xrightarrow{\simeq}     H^*_{\mathrm{cond}}(T;A).
\end{align}
Haine extended this result to locally compact Hausdorff spaces and to very general coefficients (living in the $\infty$-world) \cite[Corollary 4.12]{arXiv:2210.00186} using that the comparison map between sheaf and condensed cohomology is induced by the comparison geometric morphism
$$\CondAni_{/T}\rightarrow \Sh(T)\comma $$
which he showed to have a fully faithful pullback after Postnikov completion and tensoring with any compactly assembled $\infty$-category $\mathcal{E}$. 
\end{recollection}
Our results on the shape $\Shape^c$ provide a further extension of (\ref{cohomocomp}) to locally contractible spaces showing that we can even pass over to arbitrary constant coefficients for all CW-spaces. 
We recall the definition of cohomology internally to an $\infty$-topos which follows a general pattern, see \cites[Section 2.3;]{Nikolaus2015Coho}[Definition 7.2.2.14]{HTT}.
\begin{construction}{\textbf{Cohomology with constant coefficients}}\label[construction]{Coho_constantcoeff} \\
Let $\mathcal{X}$ be an $\infty$-topos. The cohomology of an object $X\in \mathcal{X}$ with coefficients in another object $A\in \mathcal{X}$ is defined as the set of connected components 
$$\mathcal{H}_\mathcal{X}(X,A)\colonequals \pi_0(\Hom_\mathcal{X}(X,A)) \in \Set$$
which means nothing else than taking morphisms from $X$ to $A$ in the homotopy category of $\mathcal{X}$. 
If $X=\ast\in \mathcal{X}$ is the terminal object, we refer to $$\mathcal{H}(\mathcal{X},A)\colonequals \mathcal{H}_\mathcal{X}(\ast,A)$$ simply as the \textit{cohomology of $\mathcal{X}$} with coefficients in $A$.
Via the canonical geometric morphism
\[
 \begin{tikzcd}[column sep = huge]
           \mathcal{X} \arrow[r, shift left=1ex, "f_*"] & \Ani\comma \arrow[l, shift left=.5ex, "f^*"]
           \end{tikzcd}
\]
we can assign a constant object $f^*(A)\in \mathcal{X}$ to every object $A\in \Ani$. Then the cohomology
$$\mathcal{H}_\mathcal{X}(X,A)\colonequals\mathcal{H}_\mathcal{X}(X,f^*(A))= \pi_0(\Hom_\mathcal{X}(X,f^*(A))$$ 
is the \textit{cohomology with constant coefficients} in $A$.
By left adjointness, cohomology of $X$ with constant coefficients in $A$ is computed by the shape $\Shape(\mathcal{X}_{/X})$ as 
\begin{align}\label{cohomology_shape}
\Hom_\mathcal{X}(X,f^*(A))= \Hom_{\ProAni}(\Shape(\mathcal{X}_{/X}), A).
\end{align}
For example, we can think about $A$ as the \textit{delooping} (or Eilenberg-MacLane object of degree $1$) of an abelian group $G$. This is the groupoid $A=BG\in \Ani$ with one object and the group $G$ as morphism space. It defines an abelian group object in $\Ani$ and its image $f^*(A)$ is also one in $\mathcal{X}$ as $f^*$ is finite limit preserving. If we take $B^nG$ the \textit{$n$-fold delooping} (or Eilenberg-MacLane object in degree $n$) of an abelian group, we refer to 
$$\mathcal{H}_\mathcal{X}^n(X,G)\colonequals \pi_0(\Hom_\mathcal{X}(X,f^*(B^nG))$$
as the \textit{abelian sheaf cohomology in degree $n$} of $X$ with coefficients in $G$.
In this case, the cohomology set inherits the structure of a group and we can talk about \textit{cohomology groups}.
More generally, $A$ can be chosen to be an arbitrary abelian group object in $\mathcal{X}$.
\end{construction}
This general $\infty$-topos-theoretic perspective on abelian sheaf cohomology agrees with the expression in the language of derived functors
$$\mathcal{A}\mathrm{b}(\mathcal{X})\rightarrow \Ab\Grp, \quad A\mapsto R^n\Gamma(X,A)$$ of the global sections functor of abelian sheaves $\mathcal{A}\mathrm{b}(\mathcal{X})$ in $\mathcal{X}$ which is used in \cite[Theorem 3.2]{Scholze:condensednotes} to compare condensed cohomology and sheaf cohomology of topological spaces.
\begin{remark}{\textbf{Cohomology theories of topological spaces}}\\
On topological spaces $T\in \Top$, there are three well-known cohomology theories with coefficients in abelian groups $A\in \Ab\Grp$:
\begin{itemize}
    \item[1.] \textit{Singular cohomology:} $H^n_{\mathrm{sing}}(T,A)$
    \item[2.] \textit{\v{C}ech cohomology:} $H^n_{\text{\v{C}}\mathrm{ech}}(T,A)$
    \item[3.] \textit{Sheaf cohomology:} $H^n_{\mathrm{sheaf}}(T,A)$
\end{itemize}
Sheaf cohomology can be defined following the general pattern described in \cref{Coho_constantcoeff} applied to the $\infty$-topos $\mathcal{X}=\Sh(T)$, compare \cite[Section 6.5.4 (5)]{HTT}, and can thus be formulated for all constant coefficients $A\in \Ani$.
\end{remark}
\begin{remark}\label[remark]{rem:cohomol_agree_locally_contractible}
In general, the cohomology theories on topological spaces differ. For all paracompact Hausdorff spaces the natural existing map $$H^n_{\text{\v{C}}\mathrm{ech}}(T,A) \rightarrow H^n_{\mathrm{sheaf}}(T,A)$$ is known to be an isomorphism \cite[Th\'eor\`eme 5.10.1]{Godement1973cohopara}. If $T$ is locally (weakly) contractible, there further is a natural isomorphism
$$H^n_{\text{sheaf}}(T,A) \xrightarrow{\simeq} H^n_{\mathrm{sing}}(T,A)\comma $$ see \cite[Corollary 3.29]{MR4651175}. Even more general statements can be found in \cite{Petersen_2022} and \cite{sella2016comparisonsheafcohomologysingular}.
\end{remark}
We define condensed cohomology as cohomology internally to the $\infty$-topos $\CondAni$.
\begin{definition}{\textbf{Condensed cohomology}}\label[definition]{def:cond_cohomology}\\
Let $X\in \CondAni$. The \textit{condensed cohomology} of $X$ with values in $A\in \CondAni$ is defined as 
$$H_{\mathrm{cond}}(X,A)\colonequals\mathcal{H}_{\CondAni}(X,A)\colonequals \pi_0(\Hom_{\CondAni}(X,A))\period $$ 
If $A=G$ is an abelian group object in $\CondAni$, i.e., an object in the category of condensed abelian groups $\Cond(\Ab\Grp)$, we have the \textit{condensed cohomology groups}
$$H^n_{\mathrm{cond}}(X,G)\colonequals H_{\mathrm{cond}}(X,B^nG)\comma $$
where $B^nG$ is the $n$-fold delooping or Eilenberg-MacLane object of degree $n$.
\end{definition}
By (\ref{cohomology_shape}) in \cref{Coho_constantcoeff}, for an arbitrary $\infty$-topos $\mathcal{X}$ the cohomology $\mathcal{H}(\ast,A)$ with constant coefficients $A\in \Ani$ only depends on the shape $\Shape(\mathcal{X})$. Hence, two $\infty$-topoi with the same shape also have the same constant cohomology. In particular, their cohomologies with values in abelian groups agree.
\begin{lemma}
For all topological spaces $T$ and $A\in \Ani$, there is a natural comparison map 
$$H_{\mathrm{sheaf}}(T,A) \rightarrow H_{\mathrm{cond}}(T,A) \in \Set\period $$
\end{lemma}
\begin{proof}
This follows from commutativity of the triangle of right adjoints
\begin{equation}
\begin{tikzcd}
&
\Ani  & 
\\ 
\Sh(T)\arrow[ur, "\Gamma_1"]& & \CondAni_{/T} \arrow[ul, swap, "\Gamma_2"] \arrow[ll, "c_{T,*}"]\end{tikzcd}
\end{equation}
inducing a natural transformation $\Gamma_1 \rightarrow \Gamma_2\circ c_T^*$ which is nothing else than a map
$$\Hom_{\Sh(T)}(T,A) \rightarrow \Hom_{\CondAni_{/T}}(T,A)$$
natural in $A\in \Ani$, by identifying $A$ with the corresponding constant sheaf in the $\infty$-topoi. Applying $\pi_0$ gives the cohomology result.
\end{proof}
As a consequence, for $T$ a locally contractible space, sheaf cohomology and condensed cohomology with truncated constant coefficients agree. We also recover cohomology results by Clausen-Scholze and Haine for (locally) compact Hausdorff spaces using our shape perspective.
\begin{proposition}\label[proposition]{prop:coho_cond_sheaf}
If $T\in \Top$ is locally compact Hausdorff or locally contractible, we have a natural isomorphism
$$H_{\mathrm{sheaf}}(T,A) \xrightarrow{\simeq} H_{\mathrm{cond}}(T,A) \in \Set$$
for all $A\in \Ani_{<\infty}$. 
 \end{proposition}
\begin{proof}
Recall that by \cref{prohoshape}, the shape $\Shape^c(T)$ coincides with the shape $\Shape^{\mathrm{post}}(T)$ for a locally compact Hausdorff space. Similarly, by \cref{Locally_contractible_cshape} and \cref{locallycontrshape}, the shape $\Shape^c(T)$ of a locally contractible space $T$ coincides with the shape $\Shape^{\mathrm{hyp}}(T)$. Hence, by \cref{protruncation_shapeequiv}, the shape $\Shape^c(T)$ is isomorphic to the shape $\Shape(T)$, in general, only up to pro-truncation. Thus, we need to pass to truncated coefficients. Formally, this follows from
$$\Hom_{\Pro(\Ani)}(\Shape(T),A)\simeq \Hom_{\Pro(\Ani_{<\infty})}(\Pi_{<\infty}(T),A)\simeq\Hom_{\ProAni}(\Pi_{\infty}^{\mathrm{hyp}/\mathrm{post}}(T),A)\comma$$
which holds for all $A\in \Ani_{<\infty}\subset \Pro(\Ani_{<\infty})$ using that the pro-truncation $\tau_{<\infty}\colon \ProAni \rightarrow \Pro(\Ani_{<\infty})$ is left adjoint to the inclusion of pro-truncated anima $\Pro(\Ani_{<\infty})\hookrightarrow \ProAni$.
\end{proof}
\begin{remark}
In \cref{prop:coho_cond_sheaf}, if $\Sh(T)$ is Postnikov complete in the first case (e.g., $T$ profinite space) or, respectively, hypercomplete in the second (e.g., $T$ a CW complex,  cf. \cref{locallycontrshape}), we can consider all $A\in \Ani$.   
\end{remark}
Note that we do not consider as general coefficients as Haine does in \cite[Corollary 4.12]{arXiv:2210.00186}. But in return, we extend the comparison results (\ref{cohomocomp}) to a new class of topological spaces.  
\begin{corollary}\label[corollary]{cor:coho_cech_cond}
If $T$ is a locally contractible space and $A\in \Ab\Grp$ an abelian group, we further have a natural isomorphism in $\Grp$
$$ H^n_{\mathrm{cond}}(T,A)\xrightarrow{\simeq} H^n_{\mathrm{sing}}(T,A)\period $$
\end{corollary}
\begin{proof}
This follows from the fact that the singular and sheaf cohomology coincide for locally contractible spaces and abelian coefficients $A\in \Ab\Grp$, see \cref{rem:cohomol_agree_locally_contractible}.
\end{proof}
\section{Digression: Underlying topological groups}\label[subsection]{Condensed_Homotopy_Groups}
Every condensed group has an \textit{underlying topological group} via a functor
$$(\cdot)^{\TopGrp}\colon \CondGrp\rightarrow \TopGrp\comma $$ just as every condensed set has an underlying topological space. Subsequently, we will specify this functor.
Note that this section is independent of the main part but can, however, help us better understand condensed homotopy groups of condensed anima. Therefore, keep the following picture on the homotopy-theoretic directions of condensed anima in mind.
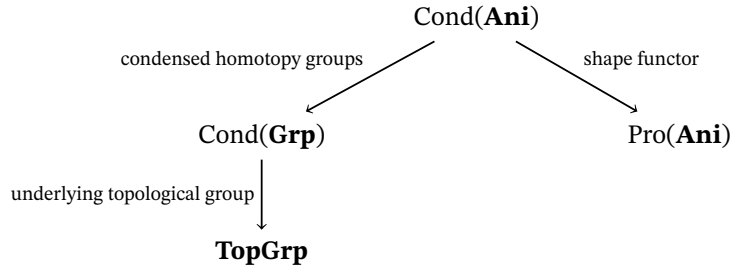
\begin{equation*}
\begin{tikzcd}[row sep=large]
&\CondAni \arrow[rightarrow, swap, dl, "\text{condensed homotopy groups}"]\arrow[rightarrow, dr, "\text{shape functor}"]& \\
\CondGrp \arrow[d,swap, "\text{underlying topological group}"]& & \ProAni \\
\TopGrp & &
\end{tikzcd}
\end{equation*}
\begin{recollection}
Condensed groups arise as group objects in $\CondSet$. Restriction of the functor $\underline{(\cdot)}\colon \Top \rightarrow \CondSet, \ T \mapsto \underline{T}$ in \cref{ex:top} to group objects induces a functor 
\begin{align}\label{functor_topcondgrp}
\underline{(\cdot)}\colon\TopGrp \rightarrow \CondGrp.
\end{align} 
The functor $\underline{(\cdot)}$ also admits a left adjoint $(\cdot)^{\Top} \colon \CondSet \rightarrow \Top, \ X \mapsto X(*)^{\Top}$ sending any condensed set $X$ to its underlying set $X(\ast)$ equipped with the quotient topology of the disjoint union over all profinite sets mapping to $X$, see \cite[Proposition 1.7]{Scholze:condensednotes}. As this functor does not preserve finite products, it does not restrict to a left adjoint of (\ref{functor_topcondgrp}). Nevertheless, by an abstract adjoint functor theorem, compare \cite[Remark 1.8]{Scholze:condensednotes}, it is still known that there exists a left adjoint functor of (\ref{functor_topcondgrp}), the \textit{underlying topological group functor} $$(\cdot)^{\TopGrp}\colon \CondGrp\rightarrow \TopGrp\period $$    
\end{recollection}
So far, the underlying topological group functor was not known in precise terms.  
We make use of the notion of quasi-topological groups to give a more concrete description.
\begin{definition}{\textbf{Quasi-topological groups}}\\
A \textit{quasi-topological group} is a group $G\in \Grp$ with a topology such that 
\begin{itemize}
    \item[(i)] the inversion $G\rightarrow G, \ g\mapsto g^{-1}$, is continuous and
    \item[(ii)] the group operation $G\times G \rightarrow G$ is continuous in each variable.
\end{itemize}
The category of quasi-topological groups with continuous group homomorphisms is $\qTopGrp$.
 \end{definition}
Note that the second condition in the definition is equivalent to the translations $g\mapsto gh$ and $h\mapsto hg$ being homeomorphisms for every $h\in G$. 
\begin{remark}\label[remark]{locallycompacttopo}
Every topological group is a quasi-topological group and there exists a fully faithful inclusion functor 
$$\TopGrp \hookrightarrow \qTopGrp\period $$
A locally compact Hausdorff quasi-topological group is a topological group~\cite[Theorem~2]{Ellis1957}.
\end{remark}
\begin{remark}
There are interesting examples of topologized groups which are not topological but quasi-topological. For instance, this phenomenon occurs when equipping fundamental groups of topological spaces in a canonical way with the quotient topology of the loop space, as for example studied by Brazas in 
\cite{Brazas2011Freeandtop}, \cite{Brazas2013topofundgr}, \cite{Brazas2015fundquot}. It is an open question to us whether it is possible to recover these from a construction internal to condensed mathematics.
\end{remark} 
\begin{proposition}{\cite[Lemma 3.2., Proposition 3.3., Corollary 3.9]{Brazas2013topofundgr}}
The fully faithful inclusion $\TopGrp \hookrightarrow \qTopGrp$ has a left adjoint $\tau \colon \qTopGrp \rightarrow \TopGrp$, such that for $G\in \qTopGrp$ the topological group $\tau(G)$ has the same underlying group structure but a coarser topology such that 
\begin{itemize}
    \item[(i)] $G$ and $\tau(G)$ share the same lattice of open subgroups,
    \item[(ii)] for every topological group $G\in \qTopGrp$, we find $G=\tau(G)\in \TopGrp$.
    \end{itemize}
\end{proposition}
Involving this left adjoint, we can describe the underlying topological group functor.
\begin{proposition}\label[proposition]{prop:underlying_top}
Restriction of the functor $$(\cdot)^{\Top}\colon\CondSet \rightarrow \Top, \quad X \mapsto X(\ast)^{\Top}$$ to group objects in $\CondSet$ lands in the category of quasi-topological groups, i.e., it restricts to 
$$\CondGrp \rightarrow \qTopGrp\period $$
Then the left adjoint functor $(\cdot)^{\TopGrp}\colon \CondGrp\rightarrow \TopGrp$ is given as the composition 
\begin{align}\label{eq:compos}
(\cdot)^{\tau}\colon\CondGrp\rightarrow \mathrm{q}\TopGrp\xrightarrow{\tau} \TopGrp\period
\end{align}
\end{proposition}
\begin{proof}
First (\textbf{1.}), we prove that restriction of the functor $$(\cdot)^{\Top}\colon\CondSet \rightarrow \Top, \quad X \mapsto X(\ast)^{\Top}$$ to group objects in $\CondSet$ lands in the category $\qTopGrp$.
Second (\textbf{2.}), we show that the described composition (\ref{eq:compos}) takes over the role as the left adjoint functor $(\cdot)^{\TopGrp}\colon \CondGrp\rightarrow \TopGrp$, i.e., for all condensed groups $G\in \CondGrp$ and topological groups $F\in \TopGrp$ there is a bijection of hom-sets natural in $G$ and $F$
\begin{align}\label{natiso1}
\Hom_{\CondGrp}(G, \underline{F})\simeq\Hom_{\TopGrp}(G^{\tau}, F)\period 
\end{align}
\textbf{(1.)} We need to show that for every condensed group $G$ the underlying group $G(\ast)$ equipped with the quotient topology carries the structure of a quasi-topological group and that the induced morphisms between the underlying groups of two condensed groups are continuous group homomorphisms. 
As $G$ is a condensed group, every set $G(K)$ carries a group structure and we have morphisms $G \rightarrow G$ of condensed sets representing the inverse and translation maps. 
More precisely, there is a morphism $(\cdot)^{-1}\colon G\rightarrow G$ of condensed sets such that for all $K$ the map $G(K) \rightarrow G(K)$ sends a group element $g_K \in G(K)$ to its inverse $(g_K)^{-1} \in G(K)$. 
Moreover, for all $h\in G(\ast)$ we have morphisms $(h\cdot)\colon G\rightarrow G$ and $(\cdot h)\colon G\rightarrow G$ sending an element $g_K \in G(K)$ to the translation $h_K\cdot g_K$ or $g_K\cdot h_K$, respectively. Here, $h_K$ denotes the image of $h \in G(\ast)$ under the map $G(\ast) \rightarrow G(K)$ induced by the unique morphism $K\rightarrow \ast$.
These morphisms are sent to continuous morphisms $G(\ast)^{\Top} \rightarrow G(\ast)^{\Top}$ on the topological space $G(\ast)^{\Top}$, whose underlying set carries the structure of a group.
Hence, $G(\ast)^{\Top}$ is a topologized group with continuous inverse $g\mapsto g^{-1}$ as well as continuous translations $g\mapsto hg$ and $g\mapsto gh$ for all $h\in G(\ast)$. By definition, it thus is a quasi-topological group.
If $G\rightarrow H$ is a morphism of condensed groups, the map $G(\ast)^{\Top} \rightarrow H(\ast)^{\Top}$ is a morphism of quasi-topological groups if it is a continuous group homomorphism.
The map is clearly continuous as it lives in the category $\Top$. It is also a group homomorphism by the definition of morphisms of condensed groups as level-wise group homomorphisms (satisfying certain compatibility conditions).\\
\textbf{(2.)} We already know about isomorphisms  natural in $C\in \CondSet$ and $T\in \Top$
\begin{align}\label{natiso2}
\Hom_{\CondSet}(C, \underline{X})\simeq \Hom_{\Top}(C(\ast)^{\Top}, X),\end{align} as well as natural in $Q\in \qTopGrp$ and $F\in \TopGrp$
\begin{align}\label{natiso3}
\Hom_{\qTopGrp}(Q,F)\simeq \Hom_{\TopGrp}(\tau(Q),F).
\end{align} If we are able to show that the isomorphism (\ref{natiso2}) restricts on group objects in $\CondSet$ and $\Top$ to an isomorphism 
\begin{align}\label{natiso4}
\Hom_{\CondGrp}(G, \underline{F})\simeq \Hom_{\qTopGrp}(G(\ast)^{\Top}, F),
\end{align} which is then automatically natural in $G\in \CondGrp$ and $F\in \TopGrp$, we can deduce the claim by combining (\ref{natiso4}) and (\ref{natiso3}) to obtain (\ref{natiso1}). Note that one has $G^{\tau}=\tau(G(\ast)^{\Top})$ by definition.
So let $G\rightarrow \underline{F}$ be a morphism of condensed groups, where $F$ is a topological group. In particular, it is a morphism on the level of condensed sets and thus under (\ref{natiso2}) corresponds to a morphism $G(\ast)^{\Top} \rightarrow F$ of topological spaces. We already argued above, that $G(\ast)^{\Top} \rightarrow \underline{F}(\ast)^{\Top}$ is a morphism of quasi-topological groups. This remains through after postcomposition with the continuous counit $\underline{F}(\ast)^{\Top} \rightarrow F$ which is on the underlying groups just the identity and thus clearly a morphism on the level of quasi-topological groups, i.e., a continuous group homomorphism. On the other hand, we have to show that every continuous group homomorphism $G(\ast)^{\Top} \rightarrow F$ corresponds under (\ref{natiso2}) to a morphism $G\rightarrow \underline{F}$ of condensed groups and not just of condensed sets. Indeed, we need to prove that
for all $K\in \Extr$ the morphism $G(K) \rightarrow \underline{F}(K)$ of sets is indeed a homomorphism of groups such that the square
\[
\begin{tikzcd}[row sep=small]
G(K') \ar[r] \ar[d] & G(K)\ar[d] \\
\underline{F}(K) \ar[r] & \underline{F}(K')
\end{tikzcd}
\]
commutes for all $f\colon K \rightarrow K'$ in $\Extr$.
The commutativity condition is already satisfied by $G\rightarrow \underline{F}$ being a morphism of condensed sets.
For the group homomorphism condition, note that the morphism $G(K)\rightarrow \underline{F}(K)$ is given as the composition $G(K)\rightarrow \underline{G(\ast)^{\Top}}(K) \rightarrow \underline{F}(K)$ of the unit of the adjunction with the map induced by the given morphism $G(\ast)^{\Top} \rightarrow F$. 
The set $ \underline{G(\ast)^{\Top}}(K)$ is in general not a group as we do not have a continuous multiplication on the quasi-topological group $G(\ast)^{\Top}$. 
Nevertheless, it formally follows that for every $K$ the composition above is a homomorphism of groups: For an element $g\in G(K)$, its image under the unit $G(K) \rightarrow \underline{G(\ast)^{\Top}}(K)=\Hom_{\Top}(K, G(\ast)^{\Top})$ is given by sending $g$ to the morphism $$K\rightarrow G(\ast)^{\Top}, \ (k=(\ast\xrightarrow{k} K)) \mapsto (G(k)(g))\period $$ 
This image is just the evaluation of the morphism $g\colon h_k \rightarrow G$, corresponding to the element $g\in G(K)$ under the Yoneda identification $G(K)=\Hom(h_k, G)$, at the point $\ast$.
If we consider two elements $g,h \in G(K)$ and their product $g\cdot h$ in $G(K)$, then by the above its image under $G(K) \rightarrow \underline{G(\ast)^{\Top}}(K)=\Hom_{\Top}(K, G(\ast)^{\Top})$ is given  by the morphism $$K\rightarrow G(\ast)^{\Top}, \ (k=(\ast\xrightarrow{k} K)) \mapsto (G(k)(g\cdot h))$$ which is continuous by definition of the topology on $G(\ast)^{\Top}$. As for all $k\colon \ast\rightarrow K$, the induced morphism $G(K)\xrightarrow{G(k)} G(\ast)$ is a group homomorphism, we have for all these $k$
$$(G(k)(g\cdot h))=G(k)(g)\cdot G(k)(h)\period $$ Further, the morphism $\underline{G(\ast)^{\Top}}(K) \rightarrow \underline{F}(K)$ is given by postcomposition with the continuous map $G(\ast)^{\Top} \rightarrow F$ of quasi-topological groups, which indeed preserves the group structure. All in all, the map  $G(K)\rightarrow \underline{G(\ast)^{\Top}}(K) \rightarrow \underline{F}(K)$ is compatible with the group structures on $G(K)$ and $\underline{F}(K)$ and thus a group homomorphism.
\end{proof}
Note that the underlying quasi-topological group will carry a compactly generated topology as this is true for the underlying topological space of every condensed set. Its image under the functor
$\tau\colon \mathrm{q}\TopGrp \rightarrow \TopGrp$
will then, in general, carry a coarser topology.
\begin{corollary}
If the underlying topological space of a condensed group $G$ is locally compact Hausdorff, the image of $G$ under 
$$(\cdot)^{\TopGrp}\colon \CondGrp\rightarrow \TopGrp$$
agrees with $G(\ast)^{\Top}$ inherited with the group structure of $G(\ast)$. 
\end{corollary}
\begin{proof}
By \cref{locallycompacttopo}, every locally compact Hausdorff quasi-topological group is already a topological group and thus $G^{\TopGrp}=\tau(G(\ast)^{\Top})=G(\ast)^{\Top}$.
\end{proof}


\DeclareFieldFormat{labelalphawidth}{#1}
\DeclareFieldFormat{shorthandwidth}{#1}
\printbibliography[heading=references]


\bigskip

\noindent \textsc{Catrin Mair, Universität Münster, Einsteinstrasse 62, 48149 Münster, Germany\\Contact:}
catrin.mair@uni-muenster.de
\medskip

\end{document}